\documentclass[12pt]{article}
\usepackage{fancyhdr}
\usepackage{bm}
\usepackage{hyperref}
\usepackage{graphicx} 
\usepackage{nicematrix}
\usepackage{amsthm,amssymb,amsmath}
\usepackage[T1]{fontenc}
\usepackage[utf8]{inputenc}
\usepackage{hyperref}
\usepackage{lipsum}
\usepackage{mathrsfs}
\usepackage{enumitem}
\usepackage{stmaryrd}
\usepackage{setspace}
\usepackage{yfonts}
\usepackage{float}
\usepackage{mathtools}
\usepackage{parskip}
\usepackage[all,cmtip]{xy}
\usepackage{tikz-cd}
\tikzcdset{row sep/normal=50pt, column sep/normal=50pt}

\newtheorem{lemma}{Lemma}[section]
\newtheorem{remark}[lemma]{Remark}
\newtheorem{theorem}[lemma]{Theorem}

\newtheorem{corollary}[lemma]{Corollary}

\newtheorem{manualtheoreminner}{Theorem}
\newenvironment{manualtheorem}[1]{%
  \renewcommand\themanualtheoreminner{#1}%
  \manualtheoreminner
}{\endmanualtheoreminner}
\usepackage{blindtext}
\usepackage[a4paper, total={6in, 8in}]{geometry}
\title{Extremal Contractions of Projective Bundles}
\author{\small{Ashima Bansal, Supravat Sarkar, Shivam Vats}}
\date{}
\begin{document}
\maketitle
\begin{abstract}
In this article, we explore the extremal contractions of several projective bundles over smooth Fano varieties of Picard rank $1$. We provide a whole class of examples of projective bundles with smooth blow-up structures, derived from the notion of drums which was introduced by Occhetta-Romano-Conde-Wiśniewski to study interaction with $\mathbb{C}^*$-actions and birational geometry. By manipulating projective bundles, we give a simple geometric construction of the rooftop flip, which was introduced recently by Barban-Franceschini. Additionally, we obtain analogues of some recent results of Vats in higher dimensions. The list of projective bundles we consider includes all globally generated bundles over projective space with first Chern class $2$. For each of them, we compute the nef and pseudoeffective cones.

\end{abstract}
\begin{center} 
\end{center}
\begin{center}
\end{center}
\section{Introduction}
The smooth projective varieties of Picard rank 2 have a relatively simple yet non-trivial structure, making them ideal cases for studying the intricate relationships between geometry and topology. Understanding these varieties can provide insights into the general properties of algebraic varieties. Varieties with low Picard rank play a crucial role in the Minimal Model Program (MMP) and Mori Theory, which are central areas of research in algebraic geometry. These theories aim to classify higher-dimensional varieties by breaking them down into simpler components.

\hspace{10pt}There are two standard ways of constructing such varieties. One is to take a projective bundle over a smooth projective variety of Picard rank $1$, and the other is to blow-up a smooth projective variety of Picard rank $1$ along a smooth subvariety. It is natural to study smooth projective varieties of Picard rank $2$ that possess both a projective bundle structure and a smooth blow-up structure. The simplest example is the blow-up of a projective space along a linear subspace of codimension at least 2. It is well-known that this blow-up has a projective bundle structure over another projective space. In \cite{Ra}, \cite{Va}, \cite{GN}, several examples of projective bundles with blow-up structures are provided. \cite{Li1} and \cite{Li2} classify varieties with both projective bundle and blow-up structures under certain conditions. 

\hspace{10pt} One of the goals of this paper is to provide further examples of projective bundles over smooth projective varieties of Picard rank $1$ which have smooth blow-up structures. The notion of a \textit{drum} was introduced in \cite{ORCW} to study the interaction between $\mathbb{C}^*$-actions and birational geometry. Using this concept of drum, in \S 3, we show that a whole class of examples can be introduced.

\textbf{Theorem A} (Theorem \ref{theorem:theorem drum}):
Let $X$ be a smooth drum constructed upon the triple $ (Y, \mathcal{L}_{-}, \mathcal{L}_{+})$. Define the vector bundles $ \mathcal{E}_{\pm}$ on $ Y_{\pm}$  by $ \mathcal{E}_{\pm}= (p_{\pm})_*(\mathcal{L}_{\mp}).$ As in {\cite [Remark 2.13]{LF}}, we have embeddings $Y_{\pm}\hookrightarrow X $, and the following hold:

$(i)$ $ \mathbb{P}_{Y_{-}}(L_{-} \oplus \mathcal{E}_{-}) \cong \textnormal{Bl}_{Y_{+}} (X)$

$(ii)$ Suppose $s$ is a nowhere vanishing of $ L_{-}\oplus \mathcal{E}_{-}$, and define the vector bundle $F$ on $ Y_{-}$ by the exact sequence 

\begin{equation*}
    0 \longrightarrow  \mathcal{O}_{Y_{-}} \xlongrightarrow{s} L_{-} \oplus \mathcal{E}_{-} \longrightarrow F \longrightarrow 0.
\end{equation*}

Then, there is a hyperplane $H_s$ in $\mathbb{P}(V)$ such that $H_{s} \cap X$, $ H_{s}\cap Y_{+}$ are smooth, $Y_{+}\not\subset H_s$ and $ \mathbb{P}_{Y_{-}}(F) \cong \textnormal{Bl}_{H_{s}\cap Y_{+}}(H_{s} \cap X).$  

As concrete applications of Theorem A, we give the following explicit examples:

\textbf{Theorem B} (Corollary \ref{theorem:theorem grassmanian} and \ref{theorem:theorem r}):
\begin{enumerate}

 \item {For 1 $\leq r \leq m-1$ } we have: 
 \item [{1.1}] The blow-up of $\textnormal{Gr}(r+1, m+2)$ along $\textnormal{Gr}(r, m+1)$ is a projective bundle over $\textnormal{Gr}(r+1, m+1)$.
 \item [{1.2}] The blow-up of $\textnormal{Gr}(r+1, m+2)$ along $\textnormal{Gr}(r+1, m+1)$ is a projective bundle over $\textnormal{Gr}(r, m+1)$.
\item {For $m\geq3$, blow-up of $\textnormal{OG}(m, 2m+1)$ along $\textnormal{OG}_{+}(m, 2m)(\cong$ $\textnormal{OG}(m-1, 2m-1)$) is a projective bundle over $\textnormal{OG}(m-1, 2m-1)$}. 
\item Let $m\geq 2$, $1\leq r \leq m-1$, and let $w$ be a skew-form on $k^{2m+1}$ of maximal rank. We have:
\item [{3.1}] The blow-up of $\textnormal{Gr}_{w}(r+1, 2m+1)$ along $\textnormal{SG}(r, 2m)$ is a projective bundle over $\textnormal{SG}(r+1, 2m)$.
\item[{3.2}] The blow-up of $\textnormal{Gr}_{w}(r+1, 2m+1)$ along $\textnormal{SG}(r+1, 2m)$ is a projective bundle over $\textnormal{SG}(r+1, 2m)$.

\item  Let $ n\geq 2$. The following statements hold: 

\item [{4.1}] $\mathbb{P}_{\mathbb{P}^n}(\mathcal{O}_{\mathbb{P}^n}(1) \oplus T_{\mathbb{P}^n}(-1))$ is the blow-up of a smooth quadric in $\mathbb{P}^{2n+1}$ along a linear subvariety of dimension $n$.
    
\item [{4.2}] If 
    \begin{equation}
0\longrightarrow \mathcal{O}_{\mathbb{P}^n}\xlongrightarrow {s}{\mathcal{O}_{\mathbb{P}^n}(1) \oplus T_{\mathbb{P}^n}(-1)}\longrightarrow F\longrightarrow 0
\end{equation}
 is an short exact sequence of vector bundles, where \lq$s$\rq\medspace is a nowhere vanishing section, then $\mathbb{P}_{\mathbb{P}^n}(F)$ is the blow-up of the smooth quadric in $\mathbb{P}^{2n}$ along a linear subvariety of dimension $n-1$.
 \item 
 Let $ n\geq 3$. The following statements hold:
     
\item[{5.1}] $\mathbb{P}_{\mathbb{P}^n}(\mathcal{O}_{\mathbb{P}^n}(1) \oplus \Omega_{\mathbb{P}^n}(2)) $ is the blow-up of $\textnormal{Gr}(2,n+2)$ along $\textnormal{Gr}(2,n+1)$.    

\item[{5.2}] If \lq$s$\rq\medspace is a nowhere vanishing section of $ \mathcal{O}_{\mathbb{P}^n}(1)\oplus \Omega_{\mathbb{P}^n}(2)$ and the vector bundle $E$ is defined by the exact sequence 

$$ 0\longrightarrow \mathcal{O}_{\mathbb{P}^n}(1) \longrightarrow \mathcal{O}_{\mathbb{P}^n}(1) \oplus \Omega_{\mathbb{P}^n}(2) \longrightarrow E \longrightarrow 0  ,$$ then $ \mathbb{P}_{\mathbb{P}^n}(E)$ is the blow-up of the $ H \cap \textrm{Gr}(2,n+2)$ along $ H \cap \textnormal{Gr}(2,n+1)$, where $ H \cap \textnormal{Gr}(2,n+2)$ is a smooth hyperplane section of $ \textnormal{Gr}(2,n+2)$ under the Plücker embedding, such that $ H\cap \textnormal{Gr}(2,n+1)$ is also smooth. For example, this gives smooth blow-up structure on $\mathbb{P}_{\mathbb{P}^{n}}(\mathcal{O}_{\mathbb{P}^{n}}(1)\oplus \mathcal{N}_{\mathbb{P}^n}(1))$, where $n$ is odd and $\mathcal{N}_{\mathbb{P}^n}$ is a null-correlation bundle on $\mathbb{P}^n.$
 
\end{enumerate}

\hspace{10pt}
In \S 4, we provide examples of projective bundles over projective spaces that are the blow-ups of singular hypersurfaces in projective spaces along linear subvarieties. This allows us to give the following generalization/analogue of the results in \cite{Va} in higher dimensions:

\textbf{Theorem C:} (Corollary \ref{theorem:theorem shivam})
Let $d, n \geq 2$ be integers, and let 
 $$V= \{\underline{x}\in \mathbb{P}^{2n+1} \,|\, x_{0}^{d-1}x_{1}+ x_{2}^{d-1}x_{3}+\ldots+ x_{2n}^{d-1}x_{2n+1}=0\}$$ 
$$ L_{0} = \{ \underline{x} \in \mathbb{P}^{2n+1} \medspace | \medspace x_{0} = x_{2} = \ldots = x_{2n}=0 \}.$$
There exists a nonempty open set $U_{n,d}$ in $\mathbb{P}^{2n+1}$ such that, for each point $\sigma\in U_{n,d}$, we have a rank $n$ vector bundle $E_{n,d,\sigma}$ over $\mathbb{P}^n$ and a hyperplane $H_\sigma$ in $\mathbb{P}^{2n+1}$ not containing $L_0.$ The linear system $|\mathcal{O}_{\mathbb{P}(E_{n,d,\sigma})}(1)|$ defines a morphism $\mathbb{P}_{\mathbb{P}^n}(E_{n,d,\sigma})\to \mathbb{P}^{2n+1}$, which identifies $\mathbb{P}_{\mathbb{P}^n}(E_{n,d,\sigma})$ as the blow-up of $V\cap H_\sigma$ along $L_0\cap H_\sigma$.

\hspace{10pt} Rooftop flips were introduced in \cite{LF}, where they used the GIT quotient to construct rooftop flips. In \S{6}, we provide a simple and geometric construction of rooftop flips using basic manipulations with projective bundles. This approach offers a simplified and more lucid version of the construction presented in \cite{LF}. One advantage of our construction is that we can also characterize when our rooftop flip is a flip or a flop.

\textbf{Theorem D:} (Theorem \ref{theorem:theorem drum})
Under the same hypothesis and notations as in Theorem A, let $W_{\pm}=\mathbb{P}_{Y_{\pm}}(\mathcal{O}\oplus \mathcal{E}_{\pm}(L_{\pm}))$, $W=\mathbb{P}_Y(\mathcal{O}\oplus \mathcal{L}_{+}\otimes \mathcal{L}_{-}),$ and $W_0=$ normalization of the projective cone over the embedding, $i:Y\xhookrightarrow{p_-\times p_+}Y_-\times Y_+\hookrightarrow \mathbb{P}(H^0(Y_-,L_-)\otimes H^0(Y_+,L_+)).$ Then, there is a rooftop flip $\psi$ modeled by $Y$, as in \cite[Definition 3.2]{LF}, with $W, W_{\pm}, W_0$ as above. If dim$Y_->$ dim$Y_+$, $\psi$ is a flip. If dim$Y_-=$ dim$Y_+$, $\psi$ is a flop.

\hspace{10pt}If $E$ is a globally generated vector bundle on $\mathbb{P}^n$ with $c_1(E)\leq n$, then $\mathbb{P}_{\mathbb{P}^n}(E)$ turns out to be a Fano manifold of Picard rank $2$. Consequently, the contraction of the other extremal ray always exists in characteristic $0$, making it an interesting object of study. This was explored in \cite{MOC}, for the case where the
rank of $E$ is 2. If $c_1(E)\leq 1$, $E$ is either trivial or a direct sum of a trivial bundle with $\mathcal{O}_{\mathbb{P}^n}(1)$ or $T_{\mathbb{P}^n}(-1)$, where $T_{\mathbb{P}^n}$ is the tangent bundle of $\mathbb{P}^n$. It is straightforward to describe the other extremal contraction and the nef and pseudoeffective cone of the projectivization of such bundles. In \cite{SU}, globally generated vector bundles on $\mathbb{P}^n$ with first Chern class $2$ are classified:

\begin{manualtheorem}{($\ast$)}\label{theorem:theorem cl}
 Let $E$ be a globally generated vector bundle of rank $r$ on $\mathbb{P}^n$. If $c_{1}(E) = 2 $ then one of the following holds:
\begin{itemize}
\item[(1)]$ E \cong \mathcal{O}_{\mathbb{P}^n}(2) \bigoplus \mathcal{O}_{\mathbb{P}^n}^{r-1}$
\item[(2)] $ E \cong \mathcal{O}_{\mathbb{P}^n}(1) \bigoplus \mathcal{O}_{\mathbb{P}^n}(1) \bigoplus \mathcal{O}_{\mathbb{P}^n}^{r-2}$
\item[(3)] $0\longrightarrow \mathcal{O}_{\mathbb{P}^n} \longrightarrow \mathcal{O}_{\mathbb{P}^n}(1) \bigoplus T_{\mathbb{P}^n}(-1)\oplus\mathcal{O}_{\mathbb{P}^n}^{r-n} \longrightarrow E \longrightarrow 0 $, with $ r\geq n$.
\item[(4)]$n=3$ and $E \cong \Omega_{\mathbb{P}^3}(2) \bigoplus \mathcal{O}_{\mathbb{P}^3}^{r-3}$, or $E \cong \mathcal{N}_{\mathbb{P}^3}(1) \bigoplus \mathcal{O}_{\mathbb{P}^3}^{r-2}$, where $ \mathcal{N}_{\mathbb{P}^3}$ is a null-correlation bundle. 
\item[(5)]$ 0 \longrightarrow \mathcal{O}_{\mathbb{P}^n}(-1) \bigoplus \mathcal{O}_{\mathbb{P}^n}(-1) \longrightarrow \mathcal{O}_{\mathbb{P}^n}^{r+2} \longrightarrow E \longrightarrow 0 $, with $ r\geq n$.
\item[(6)]$0\longrightarrow \mathcal{O}_{\mathbb{P}^n}(-2) \longrightarrow  \mathcal{O}_{\mathbb{P}^n}^{r+1} \longrightarrow E \longrightarrow 0 $, with $r\geq n$.
\end{itemize}
\end{manualtheorem}
 
\hspace{10pt}We study the other extremal contraction of their projectivizations. If $n=1$, $E$ is a direct sum of a trivial bundle with either   $\mathcal{O}_{\mathbb{P}^n}(1)^{\oplus 2}$ or $\mathcal{O}_{\mathbb{P}^n}(2)$. These two cases are analyzed in Corollary \ref{theorem:theorem sm} and Case \ref{case:case 1} of \S 7 respectively. Therefore, we assume $n\geq 2$ in other contexts. In \S $7$, we classify when these bundles have a smooth blow-up structure.

\textbf{Theorem E:} (Theorem \ref{theorem:theorem sbu})
Let $n\geq 2$. If $E$ is a globally generated vector bundle on $\mathbb{P}^n$ of rank greater than or equal to 2 and $c_{1}(E) = 2$, then $\mathbb{P}_{\mathbb{P}^n}(E)$ has a smooth blow-up structure if and only if one of the following holds:
\begin{itemize}
   
\item[$(i)$] $E\cong \mathcal{O}_{\mathbb{P}^{n}}(1)\oplus T_{\mathbb{P}^n}(-1)$. In this case, $\mathbb{P}_{\mathbb{P}^n}(E)$  is the blow-up of a smooth quadric in $\mathbb{P}^{2n+1}$ along a linear subvariety of dimension $n$.
\item[$(ii)$] {There is a short exact sequence $0\longrightarrow\mathcal{O}_{\mathbb{P}^{n}}\longrightarrow\mathcal{O}_{\mathbb{P}^{n}}(1)\oplus T_{\mathbb{P}^{n}}(-1)\longrightarrow E\longrightarrow 0$}. In this case, $\mathbb{P}_{\mathbb{P}^n}(E)$  is the blow-up of a smooth quadric in $\mathbb{P}^{2n}$ along a linear subvariety of dimension $n-1$.
\item[$(iii)$] If $n = 2$, there is a short exact sequence $0\longrightarrow \mathcal{O}_{\mathbb{P}^{n}}(-1)^{2}\xlongrightarrow{s}\mathcal{O}_{\mathbb{P}^{n}}^{r+2}\longrightarrow E\longrightarrow 0$, and $E$ does not contain a trivial line bundle as a direct summand. In this case, $2\leq r\leq 4$, and $\mathbb{P}_{\mathbb{P}^n}(E)$ is the blow-up of $\mathbb{P}^{r+1}$ along $V$, where
\begin{equation*}
V = \begin{cases}
\textnormal{twisted cubic}, & \quad \textnormal{if}\quad r = 2, \\
Z(X_{0}X_{3} - X_{1}X_{2}, X_{1}X_{4} - X_{2}X_{3}, X_{0}X_{4} - X_{2}^{2})\hookrightarrow \mathbb{P}^{4}, & \quad\textnormal{if}\quad r = 3, \\
\textnormal{image of Segre embedding}\, \mathbb{P}^{1}\times\mathbb{P}^{2}\hookrightarrow\mathbb{P}^{5}, & \quad \textnormal{if}\quad r = 4.
\end{cases}
\end{equation*} 
\end{itemize}

\hspace{10pt} The computation of nef and pseudoeffective cones of projective bundles is an active area of research. Several papers, such as  \cite{MOC}, \cite{MS}, \cite{Mi}, have contributed to this. In \S $7$, Theorem \ref{theorem:theorem cone}, we compute the nef and pseudoeffective cones of the projectivizations of the vector bundles in Theorem \ref{theorem:theorem cl}.

\section{Notations, conventions and preliminary lemmas}

Throughout this paper, we work over the field $k=\mathbb{C}$ of complex numbers for simplicity, though most of the proofs remain valid over any algebraically closed field. A \textit{variety} is an integral, separated scheme of finite type over $k$.

\hspace{10pt} For each $ n\in \mathbb{N}$, $ \mathbb{P}^n$ denotes the projective $n$-space over $k$, $Q_n$ denotes a smooth quadric hypersurface in $ \mathbb{P}^{n+1}$ . We have the tautological subbundle $\mathcal{O}_{\mathbb{P}^n}(-1)\subset\mathcal{O}_{\mathbb{P}^n}^{n+1}$, on dualizing, we obtain a surjection $\mathcal{O}_{\mathbb{P}^n}^{n+1}\to \mathcal{O}_{\mathbb{P}^n}(1)$. We call this the \textit{tautological surjection}, which induces an isomorphism on the vector spaces of global sections. In coordinates, this is given by the ordered $(n+1)$-tuple of sections of $\mathcal{O}_{\mathbb{P}^n}(1): (x_0,x_1,...,x_n).$

For odd $ n\geq 3$, we define  \textit{null-correlation bundle} on $ \mathbb{P}^n$ to be a vector bundle $ \mathcal{N}_{\mathbb{P}^n}$ fitting into a sequence 
$$ 0\longrightarrow \mathcal{O}_{\mathbb{P}^n} \xlongrightarrow{s} \Omega_{\mathbb{P}^n}(2) \longrightarrow \mathcal{N}_{\mathbb{P}^n}(1) \longrightarrow 0,$$ where \lq$s$\rq\medspace is a nowhere vanishing section of $ \Omega_{\mathbb{P}^n}(2)$. This definition is dual to the one given in {\cite[section 4.2]{OSS}}.

Let Gr($r,n)$, OG($r,n)$, OG$_+(r,n)$, SG($r,n)$, Gr$_{w}(r, 2m+1)$ be the Grassmanian, orthogonal Grassmanian, symplectic Grassmanian, and odd symplectic Grassmanian, respectively, as in \cite{Pa} or \cite[Section 4]{CRC}. Writing $k^{m+2}=k^{m+1} \oplus k$, we see that Gr$(r, m+1)$ and Gr$(r+1, m+1)$ are naturally subvarieties of Gr$(r+1, m+2)$ for $m\geq 3$ and $1\leq r\leq m-1$. Given an orthogonal decomposition $k^{2m+1}=k^{2m} \oplus k$, we see that OG$_+(m, 2m)$ is naturally a subvariety of OG$(m, 2m+1)$ for $m\geq 3$. Finally, given a skew-form $w$ of maximal rank on $k^{2m+1}$, and writing $k^{2m+1}=k^{2m} \oplus k$, where the direct summand $k$ is the kernel of $w$, we see that SG$(r, 2m)$, SG$(r+1, 2m)$ are naturally subvarieties of Gr$_w(m, 2m+1)$ for $m\geq 2$ and $1\leq r \leq m-1$.

\hspace{10pt} If $E$ is a vector bundle over $ \mathbb{P}^n$, let rk $E$ denote the rank of $E$, and $c_1(E)$, the first Chern class of $E$, be the unique integer $d$ such that det $E=\mathcal{O}_{\mathbb{P}^{n}}(d)$. Clearly, this is the usual definition of the first Chern class if we identify $\textnormal{A}^1(\mathbb{P}^n)$ with $\mathbb{Z}$ by taking the class of a hyperplane as the generator.

\hspace{10pt}We follow the convention for projective bundles as in {\cite [Chapter 2, Section 7]{Ha}}.

\begin{lemma}
    If $n\geq 2$ and $E$ is a globally generated vector bundle on $\mathbb{P}^n$ of rank $r$ with $c_1(E)\leq n$, then $\mathbb{P}_{\mathbb{P}^n}(E)$ is a Fano variety.
\end{lemma}
\begin{proof}
If $H$ is pullback of $\mathcal{O}_{\mathbb{P}^{n}}(1)$ to $\mathbb{P}_{\mathbb{P}^n}(E)$. Then $H$ is nef. Since $E$ is globally generated, $\mathcal{O}_{\mathbb{P}(E)}(1)$ is also globally generated, hence nef. If $d=c_1(E)\leq n$, by the formula for the canonical divisor of a projective bundle (see \cite{Re}, p. 349), we have $-K_{\mathbb{P}(E)}=(n+1-d)H+r\medspace\mathcal{O}_{\mathbb{P}(E)}(1)$, which lies in the interior of the nef cone of $\mathbb{P}_{\mathbb{P}^n}(E)$ and is therefore, ample.
\end{proof}

\hspace{10pt} A proper map of varieties $p:X\to Y$ is called a \textit{contraction} if $p_*\mathcal{O}_X=\mathcal{O}_Y.$ An \textit{extremal contraction} of a smooth projective variety is a contraction of a $K$-negative extremal ray of the Mori cone. By the above lemma,
$\mathbb{P}_{\mathbb{P}^n}(E)$ is a smooth Fano variety of Picard rank $2$, so it has an extremal contraction in addition to the projective bundle contraction. We refer to this contraction as the \textit{other contraction} throughout the paper.

\hspace{10pt}If $Z$ is a subvariety of a variety $X$, we denote the blow up of $X$ along $Z$ by Bl$_Z(X)$. By a \textit{smooth blow-up}, we mean blow-up of a smooth variety along a smooth subvariety. If $p$ is a smooth blow-up, Ex$(p)$ denotes the exceptional divisor of $p$.

The following lemma is a relative version of the well-known fact in \cite[Proposition 9.11]{EH}: If $r,d \geq 1$, and $L\hookrightarrow \mathbb{P}^{r+d}$ is a linear subspace of dimension $r-1$, then $\textnormal{Bl}_\textnormal{L} \mathbb{P}^{r+d} \cong \mathbb{P}_{\mathbb{P}^d}(\mathcal{O}_{\mathbb{P}^{d}}^{r} \oplus \mathcal{O}_{\mathbb{P}^{d}}(1))$.

\begin{lemma}\label{lemma:lemma relative}
Let $ r \geq 1$, $X$ be a smooth projective variety, and $E$ a vector bundle of rank greater than equal to 2 over $X$. Identify $ X \times \mathbb{P}^{r-1} = \mathbb{P}_{X}(\mathcal{O}_{X}^{r}) $ as a projective subbundle of $ \mathbb{P}_{X}(\mathcal{O}_{X}^{r}\oplus E)$ via the surjection $\mathcal{O}_{X}^{r} \oplus E \twoheadrightarrow 
 \mathcal{O}_{X}^{r}$. Then  $ \textnormal{Bl}_{{X\times \mathbb{P}^{r-1}}}( \mathbb{P}_{X}(\mathcal{O}_{X}^{r}\oplus E)) \cong \mathbb{P}_{\mathbb{P}(E)}(\mathcal{O}_{\mathbb{P}(E)}^{r} \oplus \mathcal{O}_{\mathbb{P}(E)}(1))$, with the exceptional divisor $ \mathbb{P}(E) \times \mathbb{P}^{r-1}= \mathbb{P}_{\mathbb{P}(E)}(\mathcal{O}_{\mathbb{P}(E)}^{r}) \hookrightarrow \mathbb{P}_{\mathbb{P}(E)}(\mathcal{O}_{\mathbb{P}(E)}^{r} \oplus \mathcal{O}_{\mathbb{P}(E)}(1)).$  Under the blow-up map $ \mathbb{P}_{Y}(\mathcal{O}_{Y}^{r}\oplus \mathcal{O}_{Y}(1)) \longrightarrow \mathbb{P}_{X}(\mathcal{O}_X^{r}\oplus E)$, the line bundle $\mathcal{O}_{\mathbb{P}_{X}(\mathcal{O}_X^{r}\oplus E)}(1)$ pulls back to $\mathcal{O}_{\mathbb{P}_{Y}(\mathcal{O}_Y^{r}\oplus \mathcal{O}_Y(1) )}(1)$, where  $Y=\mathbb{P}(E).$
\end{lemma}

\begin{proof}
Let $ \pi: \mathbb{P}(E) \longrightarrow X$ be the projection. We have $ \pi_{\ast}(\mathcal{O}_{\mathbb{P}(E)}^{r} \oplus \mathcal{O}_{\mathbb{P}(E)}(1))= \mathcal{O}_{X}^{r}\oplus E$ and $ \pi_{\ast}\mathcal{O}_{\mathbb{P}(E)}^{r} = \mathcal{O}_{X}^{r}$  This gives the following commutative diagram:  
\begin{center}
\begin{tikzcd}
\mathbb{P}(E)\times \mathbb{P}^{r-1} \arrow[r, "\pi \times id"] \arrow[d, hook]
& X\times \mathbb{P}^{r-1} \arrow[d, hook ] \\
\mathbb{P}_{\mathbb{P}(E)}(\mathcal{O}_{\mathbb{P}(E)}^r \oplus \mathcal{O}_{\mathbb{P}(E)}(1)) \arrow[r, "\tilde{\pi}" ] \arrow[d, ]
& |[, rotate=0]|  \mathbb{P}_{X}(\mathcal{O}_{X}^{r}\oplus E) \arrow[d, ]\\
\mathbb{P}(E) \arrow[r, "\pi"] & X.
\end{tikzcd}
\end{center}

Under $\tilde{\pi}$, the line bundle $\mathcal{O}_{\mathbb{P}(\mathcal{O}_X^{r}\oplus E)}(1)$ pulls back to $\mathcal{O}_{\mathbb{P}(\mathcal{O}_Y^{r}\oplus \mathcal{O}_Y(1) )}(1)$, where $Y=\mathbb{P}(E).$ We want to apply  {\cite[Theorem 1.1]{ESB}}, to show that $ \tilde{\pi}$ is the blow-up along $ X\times \mathbb{P}^{r-1}$. It suffices to show for $ z \in \mathbb{P}_{X}(\mathcal{O}_{X}^{r} \oplus E)$ that:

$(i)$ $ \tilde{\pi}^{-1}(z)$ is a point for $ z \notin X \times \mathbb{P}^{r-1}$.

$ (ii)$ $ \tilde{\pi}^{-1}(z) \subseteq \mathbb{P}(E) \times \mathbb{P}^{r-1}$ for  $ z \in X \times \mathbb{P}^{r-1}$.

It suffices to check conditions $(i)$ and $ (ii)$ over each point of $X$. Over each point of $X$, $\tilde{\pi} $ is a blow-up of a projective space along a linear subvariety, and thus $(i)$ and $ (ii)$ follows. Note that $(ii)$ implies the top square in the commutative diagram is a fibre square, so the statement about the exceptional divisor follows.    
\end{proof}

\vspace{10pt}

\begin{lemma}\label{theorem:theorem key}
Let $E$ be a vector bundle over $\mathbb{P}^{n}$, and let  $\mathcal{O}_{\mathbb{P}^n}^{m+1} \oplus \mathcal{O}_{\mathbb{P}^n}(1) \xlongrightarrow{\phi} E
$ be a surjection such that $\phi(\mathcal{O}_{\mathbb{P}^n}(1)) \neq 0$. Then there exist a morphism $ \mathbb{P}_{\mathbb{P}^n}(\mathcal{O}_{\mathbb{P}^n}^{m+1} \oplus \mathcal{O}_{\mathbb{P}^n}(1))\xlongrightarrow{p}\mathbb{P}^{n+m+1}$, which is blow-up along an $m-$dimensional linear subspace $L$. Then morphism $\pi:\mathbb{P}_{\mathbb{P}^n}(E) \hookrightarrow \mathbb{P}_{\mathbb{P}^n}(\mathcal{O}_{\mathbb{P}^n}^{m+1} \oplus \mathcal{O}_{\mathbb{P}^n}(1)) \xlongrightarrow{p}\mathbb{P}^{n+m+1} $ is the blow-up of $ \pi (\mathbb{P}_{\mathbb{P}^n}(E))$ along $ \pi(\mathbb{P}_{\mathbb{P}^n}(E)) \cap L $.

\end{lemma}

\begin{proof}
By Lemma \ref{lemma:lemma relative}, with $X$ a point, the exceptional divisor of $p$, i.e., Ex(p) = $\mathbb{P}_{\mathbb{P}^n}( \mathcal{O}_{\mathbb{P}^n}^{m+1}) \hookrightarrow\mathbb{P}_{\mathbb{P}^n}(\mathcal{O}_{\mathbb{P}^n}^{m+1} \oplus \mathcal{O}_{\mathbb{P}^n}(1)) $. Since $ \phi(\mathcal{O}_{\mathbb{P}^n}(1)) \neq 0$, $\phi $ does not factor through $ \mathcal{O}_{\mathbb{P}^n}^{m+1} \oplus \mathcal{O}_{\mathbb{P}^n}(1) \longrightarrow \mathcal{O}_{\mathbb{P}^n}^{m+1}$. This implies that $\mathbb{P}_{\mathbb{P}^n}(E)$ is not contained in Ex(p). Therefore, $ \pi $ identifies $ \mathbb{P}_{\mathbb{P}^n}(E)$ with the strict transform of $ \pi(\mathbb{P}_{\mathbb{P}^n}(E))$ under the blow-up $p$, i.e., the blow-up of $\pi(\mathbb{P}_{\mathbb{P}^n}(E))$ along $ \pi(\mathbb{P}_{\mathbb{P}^n}(E)) \cap L$.
\end{proof}
\hspace{10pt} We shall need the following observation.
\begin{lemma} \label{theorem:sumtriv}
Let $E$ be a vector bundle over $\mathbb{P}^{n}$, and let $ \phi: \mathcal{O}_{\mathbb{P}^n}^{N+1} \longrightarrow E $ be a surjection, inducing $ \pi: \mathbb{P}_{\mathbb{P}^n}(E)\hookrightarrow \mathbb{P}^{n} \times \mathbb{P}^{N} \longrightarrow \mathbb{P}^{N}$. Define
$\tilde{\pi}: \mathbb{P}_{\mathbb{P}^n}(\mathcal{O}_{\mathbb{P}^n} \oplus E) \hookrightarrow \mathbb{P}^n \times \mathbb{P}^{N+1} \longrightarrow \mathbb{P}^{N+1}$ as the map induced by the surjection $\textnormal{Id}\oplus \phi: \mathcal{O}_{\mathbb{P}^n}^{N+2} \longrightarrow \mathcal{O}_{\mathbb{P}^n}\oplus E$. We have $\emph{Im}\, \tilde{\pi} = C (\emph{Im} \,\pi)$, and $ \tilde{\pi}^{-1}([t: \underline{x}])\cong \pi^{-1}([\underline{x}])$ for all $[t:\underline{x}]\in\mathbb{P}^{N+1}$ with $\underline{x}\neq 0$, and $ \tilde{\pi}^{-1}([1: \underline{0}])\cong\mathbb{P}^n$ . Here, for a variety $Z$ in $\mathbb{P}^N$, $C(Z)$ is the variety in $\mathbb{P}^{N+1}$ defined as the cone over $Z$ with vertex $[1:0:0...:0].$
 \end{lemma}
 
\begin{proof}
 Let $ \mathbb{P}^n \times \mathbb{P}^{N+1} \xlongrightarrow{f} \mathbb{P}^n $  be the projection map. For any $\alpha\in\mathbb{P}^n$, we identify $E_{\underline{\alpha}}^*$ as a subspace of $k^{N+1}$ via $\phi^*$.  For any $[t: \underline{x}] \in \mathbb{P}^{N+1}$, we have
 $$\tilde{\pi}^{-1}([t :\underline{x}]) \cong f \tilde{\pi}^{-1}( [t:\underline{x}])= \{ \underline{\alpha}\in \mathbb{P}^n\, |\, k\oplus E_{\underline{\alpha}}^{\ast} \ni (t,\underline{x}) \} $$ 
 
 \begin{equation*}
=\{ \underline{\alpha} \in \mathbb{P}^n \medspace | \medspace E_{\underline{\alpha}} ^{\ast} \ni \underline{x} \} = 
\begin{cases}
f(\pi^{-1}([\underline{x}]))\cong\pi^{-1}([\underline{x}]) , & \quad \textnormal{if}\quad \underline{x} \neq 0, \\
\mathbb{P}^{n}, & \quad\textnormal{if}\quad \underline{x} = 0.
\end{cases}
\end{equation*}
So, if $\underline{x}\neq 0$
$$ \tilde{\pi}^{-1}[t:\underline{x}] \neq  \varnothing \Leftrightarrow \pi^{-1}([\underline{x}]) \neq \varnothing  \Leftrightarrow [\underline{x}] \in \text{Im}\,\pi  \Leftrightarrow [t: \underline{x}] \in C (\text{Im}\, \pi).$$
Hence, we conclude that $ \text{Im} \,\tilde{\pi} = C(\text{Im}\, \pi)$.
\end{proof}

\begin{lemma}\label{theorem:theorem E}
Let $X$ be a smooth projective variety, and let $ Z\hookrightarrow X$ be a smooth subvariety of codimension greater than equal to 2 in $X$. Let $ \Tilde{X}\longrightarrow X $ denote the blow-up of $X$ along  $Z$, with $ E\hookrightarrow \tilde{X}$ the exceptional divisor. Then $ \mathcal{O}_{\Tilde{X}}(E)$ is not big.
\end{lemma}
\begin{proof}
By \cite[Lemma 7.11]{De}, we have $ h^{0}(\Tilde{X}, \mathcal{O}_{\tilde{X}}(mE))= h^{0}(X, \mathcal{O}_{X})=1$ for all integers $ m\geq 0$. This implies  that $ \mathcal{O}_{\Tilde{X}}(E)$ cannot be big.
\end{proof}

\hspace{10pt}For a smooth projective variety $X$, let {Nef}$(X)$ and $\overline{\textnormal{Eff}}^{1}(X)$ denote the nef and pseudoeffective cones of $X$, respectively. For $a,b\in \textnormal{N}^{1}(X)_{\mathbb{R}}$, let $<a,b>$ denote the cone generated by $a,b$.

\hspace{10pt}Let $ n\geq 2$, and let $E$ be a globally generated vector bundle on $ \mathbb{P}^n$ of rank $ r(\geq 2)$. Consider the projection $ \mathbb{P}_{\mathbb{P}^{n}}(E)\xlongrightarrow{p} \mathbb{P}^n $, and let $H= p^{\ast}(\mathcal{O}_{\mathbb{P}^n}(1)).$  Then $H$ is nef but not big, so one ray in $ \overline{\textnormal{Eff}}^{1}(\mathbb{P}_{\mathbb{P}^{n}}(E))$ is generated by $H$. Since $ \mathcal{O}_{\mathbb{P}(E)}(1)$ is globally generated, it is nef. Therefore, the other extremal ray of $ \overline{\textnormal{Eff}}^{1}(\mathbb{P}(E))$ will be generated by $ \mathcal{O}_{\mathbb{P}(E)}(1)- c H$ for some real number $c\geq 0$.

\begin{lemma}\label{theorem:theorem osum}
In above notation, suppose $ \overline{\textnormal{Eff}}^{1}(\mathbb{P}_{\mathbb{P}^{n}}(E))=<H, \mathcal{O}_{\mathbb{P}(E)}(1)- c H >$. Then $ \overline{\textnormal{Eff}}^{1}(\mathbb{P}_{\mathbb{P}^{n}}(\mathcal{O}_{\mathbb{P}^n} \oplus E))=<H, \mathcal{O}_{\mathbb{P}(\mathcal{O}_{\mathbb{P}^n} \oplus E)}(1)- cH>$. Here, by abuse of notations, $H$ also denotes the pullback of $\mathcal{O}_{\mathbb{P}^n}(1) $ on $ \mathbb{P}_{\mathbb{P}^{n}}(\mathcal{O}_{\mathbb{P}^n} \oplus E).$
\end{lemma}

\begin{proof}
    For $ a,b \in \mathbb{N},$ note that
    $$ H^{0}(\mathbb{P}_{\mathbb{P}^{n}}(E), \mathcal{O}_{\mathbb{P}(E)}(1)^{a}\otimes H^{-b})= H^{0}(\mathbb{P}_{\mathbb{P}^{n}}(E), p^{\ast} \mathcal{O}_{\mathbb{P}^n}(-b) \otimes \mathcal{O}_{\mathbb{P}(E)}(a))$$
\hspace{193pt} $ = H^{0}(\mathbb{P}^n, p_{\ast}(p^{\ast}\mathcal{O}_{\mathbb{P}^{n}}(-b) \otimes \mathcal{O}_{\mathbb{P}(E)}(a) ))$

\vspace{2pt}
\hspace{193pt} $ = H^{0}(\mathbb{P}^n, p_{\ast}(\mathcal{O}_{\mathbb{P}(E)}(a))(-b))$
\vspace{2pt}

\hspace{193pt} $= H^{0}(\mathbb{P}^n, \textrm{Sym}^{a}(E)(-b)).$

Similarly, we have 

$$ H^{0}(\mathbb{P}(\mathcal{O}_{\mathbb{P}^n} \oplus E), \mathcal{O}_{\mathbb{P}(\mathcal{O}_{\mathbb{P}^{n}}\oplus E)}(1)^{a}\otimes H^{-b})= H^{0}(\mathbb{P}^n, \textrm{Sym}^{a}(\mathcal{O}_{\mathbb{P}^n}\oplus E)(-b)) .$$

Let $a,b\in \mathbb{N}$ with $ \frac{b}{a}>c$, so  for all $ d\in \mathbb{N}$, $ d \big(a\cdot \mathcal{O}_{\mathbb{P}(E)}(1)- b H \big) \notin \overline{\textnormal{Eff}}^{1}(\mathbb{P}_{\mathbb{P}^{n}}(E))$. Hence, we have 
$H^{0}(\mathbb{P}_{\mathbb{P}^{n}}(E), \mathcal{O}_{\mathbb{P}(E)}(1)^{da}\otimes H^{-db})=0$, so $H^{0}(\mathbb{P}^n, \textrm{Sym}^{da}(E)(-db))=0 .$

Note that

$$ \textrm{Sym}^{da}(\mathcal{O}_{\mathbb{P}^n}\oplus E) = \bigoplus_{i=0}^{da} \textrm{Sym}^{i}(E),$$ so we get 
$$H^{0}(\mathbb{P}(\mathcal{O}_{\mathbb{P}^n} \oplus E), \mathcal{O}_{\mathbb{P}(\mathcal{O}_{\mathbb{P}^{n}}\oplus E)}(1)^{da}\otimes H^{-db})= H^{0}(\mathbb{P}^n, \textrm{Sym}^{da}(\mathcal{O}_{\mathbb{P}^n}\oplus E)(-d b)) $$
\hspace{232pt}$= \bigoplus_{i=0}^{da} H^{0}(\mathbb{P}^n, \textrm{Sym}^{i}(E)(-db))$
    
\vspace{2pt}
    \hspace{232pt}$ = \bigoplus_{i=0}^{da} H^{0}(\mathbb{P}_{\mathbb{P}^{n}}(E), \mathcal{O}_{\mathbb{P}(E)}(1)^{i}\otimes H^{-db})$
    
\vspace{2pt}
    \hspace{232pt}$ =0$

as $i \cdot \big( \mathcal{O}_{\mathbb{P}(E)}(1) - db \cdot H \big)$ is not effective, either $ i=0$ or for $ \frac{db}{i} \geq \frac{db}{da}= \frac{b}{a}>c.$ So $ a\mathcal{O}_{\mathbb{P}(\mathcal{O}_{\mathbb{P}^n}\oplus E)}(1) - b H$  is not big whenever $ \frac{b}{a}>c$. Hence,
    $$ \overline{\textnormal{Eff}}^{1}(\mathbb{P}_{\mathbb{P}^{n}}(\mathcal{O}_{\mathbb{P}^n} \oplus E)) \subset \langle H, \mathcal{O}_{\mathbb{P}(\mathcal{O}_{\mathbb{P}^n} \oplus E)}(1)- cH \rangle $$If $c=0$, then the lemma follows. So now assume $c>0$ and let $ a,b \in \mathbb{N} $ with $ \frac{b}{a}< c$.  As before, for all  $d \in \mathbb{N} $ we have

 \hspace{15pt}   $H^{0}(\mathbb{P}_{\mathbb{P}^{n}}(\mathcal{O}_{\mathbb{P}^n} \oplus E), \mathcal{O}_{\mathbb{P}(\mathcal{O}_{\mathbb{P}^{n}}\oplus E)}(1)^{da}\otimes H^{-db})=\bigoplus_{i=0}^{da} H^{0}(\mathbb{P}_{\mathbb{P}^{n}}(E), \mathcal{O}_{\mathbb{P}(E)}(1)^{i}\otimes H^{-db})$. 
 
 So we have 
$$ \textrm{dim}\, \medspace H^{0}(\mathbb{P}_{\mathbb{P}^{n}}(\mathcal{O}_{\mathbb{P}^n} \oplus E), \mathcal{O}_{\mathbb{P}(\mathcal{O}_{\mathbb{P}^{n}}\oplus E)}(1)^{da}\otimes H^{-db}) \geq \textrm{dim}\,  \medspace H^{0}(\mathbb{P}_{\mathbb{P}^{n}}(E), \mathcal{O}_{{\mathbb{P}}(E)}(1)^{da}\otimes H^{-db})>0$$
for $d$ sufficiently large, as $ a \cdot \mathcal{O}_{\mathbb{P}(E)}(1) - b \cdot H$ is big for $ \frac{b}{a} <c$. Hence, $a \cdot \mathcal{O}_{\mathbb{P}(\mathcal{O}_{\mathbb{P}^n} \oplus E)}(1) - b \cdot H$ is pseudoeffective whenever $ \frac{b}{a} <c$. This shows that  $$ \overline{\textnormal{Eff}}^{1}(\mathbb{P}_{\mathbb{P}^{n}}(\mathcal{O}_{\mathbb{P}^n} \oplus E)) \supset \langle H, \mathcal{O}_{\mathbb{P}(\mathcal{O}_{\mathbb{P}^n} \oplus E)}(1)- cH \rangle .$$ 
Hence, 
$$  \overline{\textnormal{Eff}}^{1}(\mathbb{P}_{\mathbb{P}^{n}}(\mathcal{O}_{\mathbb{P}^n} \oplus E)) = \langle H, \mathcal{O}_{\mathbb{P}(\mathcal{O}_{\mathbb{P}^n} \oplus E)}(1)- cH \rangle.$$
\end{proof}

\section{Projective bundle as smooth blow up}

 We follow the definition of drum and the notations as in {\cite [Section 2.3]{LF}}. Let $V_{\pm}=H^0(Y_{\pm},L_{\pm})$, and $V=V_+\oplus V_-$. By definition, $X$ is a subvariety of $\mathbb{P}(V)$. Here is the main theorem of this section.

 \begin{theorem}\label{theorem:theorem drum}
Let $X$ be a smooth drum constructed upon the triple $ (Y, \mathcal{L}_{-}, \mathcal{L}_{+})$. Define the vector bundles $ \mathcal{E}_{\pm}$ on $ Y_{\pm}$  by $ \mathcal{E}_{\pm}= (p_{\pm})_*(\mathcal{L}_{\mp}).$ As in \textnormal{{\cite [Remark 2.13]{LF}}}, we have embeddings $Y_{\pm}\hookrightarrow X $. We have 

$(i)$ $ \mathbb{P}_{Y_{-}}(L_{-} \oplus \mathcal{E}_{-}) \cong \textnormal{Bl}_{Y_{+}} (X)$

$(ii)$ Suppose $s$ is a nowhere vanishing of $ L_{-}\oplus \mathcal{E}_{-}$, and define the vector bundle $F$ on $ Y_{-}$ by the exact sequence 

\begin{equation*}
0 \longrightarrow  \mathcal{O}_{Y_{-}} \xlongrightarrow{s} L_{-} \oplus \mathcal{E}_{-} \longrightarrow F \longrightarrow 0.
\end{equation*}

Then there is a hyperplane $H_s$ in $\mathbb{P}(V)$ such that $H_{s} \cap X$, $ H_{s}\cap Y_{+}$ are smooth, $Y_{+}\not\subset H_s$ and $ \mathbb{P}_{Y_{-}}(F) \cong \emph{Bl}_{H_{s}\cap Y_{+}}(H_{s} \cap X).$  
       
 \end{theorem}

    \begin{proof}
$(i)$ Let $\mathcal{E}= \mathcal{L}_{-}\oplus \mathcal{L}_{+}$, a vector bundle of rank $2$ on $Y$. By the proof of \cite[Lemma 4.4]{ORCW}, $ \phi $ is the blow-up along $ Y_{+} \bigsqcup Y_{-}$. Let $ \widetilde{Y_{\mp}} \cong \textnormal{Bl}_{Y_{\pm}}(X).$ So we have morphisms $ \mathbb{P}(\mathcal{E}) \xlongrightarrow{q_{\pm}} \widetilde{Y_{\pm}}\xlongrightarrow{\pi_{\pm}} X$, such that $ q_{\pm}$ blows down $E_{\pm}$ to subvariety of $ \widetilde{Y_{\pm}}$ isomorphic to $ Y_{\pm}$. As in {\cite [Section 4]{ORCW}}, $ \overline{\textrm{NE}}(\mathbb{P}(\mathcal{E}))$ is simplicial, so $ \mathbb{P}(\mathcal{E})$ has exactly three elementary contractions: $ \mathbb{P}(\mathcal{E} )\xlongrightarrow{q_{+}} \widetilde{Y_{+}}$, $ \mathbb{P}(\mathcal{E}) \xlongrightarrow{q_{-}} \widetilde{Y_{-}},$ and $ \mathbb{P}(\mathcal{E}) \xlongrightarrow{\pi} Y.$ It suffices to show that $ \mathbb{P}(\mathcal{E})$ is a smooth blow-up of $ \mathbb{P}_{Y_{-}}(L_{-}\oplus \mathcal{E_{-}})$ with exceptional divisor $E_{-},$ as that will imply $ \mathbb{P}_{Y}(L_{-}\oplus \mathcal{E}_{-}) \cong \widetilde{Y_{-}}.$  

If $\mathcal{E}^{'}$ is a vector bundle on $Y_{-}$ such that $ Y \cong \mathbb{P}_{Y_{-}}(\mathcal{E}^{'})$ over $Y_{-}$, then $ \mathcal{L}_{+}= \mathcal{O}_{\mathbb{P}(\mathcal{E^{'}})}(1) \otimes (p_{-})^{\ast} M$ for some line bundle $M \in \textrm{Pic}(Y_{-})$, as $\textrm{deg}(\mathcal{L}_{\mid_{F_{-}}})=1$ for a fiber $ F_{-}$ of $p_{-}$. Thus, 
$$ \mathcal{E}_{-}= (p_{-})_{\ast}(\mathcal{L}_{+})= (p_{-})_{\ast}\mathcal{O}_{\mathbb{P}(\mathcal{E^{'}})}(1) \otimes M= \mathcal{E}^{'} \otimes M. $$ So, $ \mathbb{P}_{Y_{-}}(\mathcal{E_{-}})= Y.$ If $\mathcal{O}_{\mathbb{P}(\mathcal{E_-})}(1)= \mathcal{L}_{+}\otimes(p_{-})^{\ast} M$, then applying $(p_-)_*$ we get $\mathcal{E_-}=\mathcal{E_-}\otimes M$. Taking determinant bundles, we see that some positive power of $M$ is trivial. Since the Picard group of $Y_-$ is generated by $L_-$, $M$ must be a power of the very ample line bundle $L_-$, and since some positive power of $M$ is trivial, $M$ must be trivial. Thus, $\mathcal{O}_{\mathbb{P}(\mathcal{E_-})}(1)= \mathcal{L}_{+}.$ We have the embedding $$ Y_{-}= \mathbb{P}_{Y_{-}}(\mathcal{O}_{Y}) \hookrightarrow \mathbb{P}_{Y_{-}}(\mathcal{O}_{Y}\oplus \mathcal{E}_{-}(L^{-1})) = \mathbb{P}_{Y_{-}}(L_{-}\oplus \mathcal{E}_{-}).$$ By Lemma \ref{lemma:lemma relative}, $$ \textnormal{Bl}_{Y_{-}}(\mathbb{P}_{Y_{-}}(L_{-} \oplus \mathcal{E}_{-}))= \mathbb{P}_{\mathbb{P}(\mathcal{E}_{-}(L_{-}^{-1}))}(\mathcal{O} \oplus \mathcal{O}_{\mathbb{P}(\mathcal{E}_{-}(L_{-}^{-1}))}(1)) $$ $$ \hspace{70pt}=\mathbb{P}_{\mathbb{P}(\mathcal{E_{-}})}( \mathcal{O}\oplus \mathcal{O}_{\mathbb{P}(\mathcal{E_{-}})}(1)(\mathcal{L}^{-1})) = \mathbb{P}_{Y}(\mathcal{L}_{-} \oplus \mathcal{L}_{+})= \mathbb{P}_{Y}(\mathcal{E}),$$ and the exceptional divisor is $$ \mathbb{P}_{Y}(\mathcal{O})\hookrightarrow \mathbb{P}_{Y}(\mathcal{O}\oplus \mathcal{L}_{+}(\mathcal{L}_{-}^{-1})),$$
i.e.,
$$ \mathbb{P}_{Y}(\mathcal{L}_{-}) \hookrightarrow \mathbb{P}_{Y}(\mathcal{L}_{-}\oplus \mathcal{L}_{+})$$
   i.e., 
   $$E_{-}\hookrightarrow \mathbb{P}(\mathcal{E)}.$$  
It follows that the blow-up map $ \mathbb{P}(\mathcal{E}) \longrightarrow \mathbb{P}_{Y_{-}}(L_{-}\oplus \mathcal{E}_{-})$ is just the map $ q_{-}.$ Thus, $ \widetilde{Y_{-}} = \mathbb{P}_{Y_{-}}(L_{-}\oplus \mathcal{E}_{-}).$ So,
\begin{equation*}
\textnormal{Bl}_{Y_{+}} (X)\cong \mathbb{P}_{Y_{-}}(L_{-}\oplus \mathcal{E}_{-}).  
\end{equation*}

so $(i)$ is proved.

$(ii)$ Let $\mathcal{G}_{-}= L_{-}\oplus \mathcal{E}_{-}.$ By Lemma \ref{lemma:lemma relative}, under the blow-up map
\begin{equation*}
\mathbb{P}_{\mathbb{P}(\mathcal{E}_{-}(\mathcal{L}_{-}^{-1}))}(\mathcal{O}\oplus \mathcal{O}_{\mathbb{P}(\mathcal{E}_{-}(\mathcal{L}_{-}^{-1}))}(1)) \longrightarrow \mathbb{P}_{Y_{-}}(\mathcal{O}\oplus \mathcal{E}_{-}(L_{-}^{-1})),
 \end{equation*}

$\mathcal{O}(1)$ pulls back to $ \mathcal{O}(1)$. Since $ L_{-}$ pulls back to $ \mathcal{L}_{-}$ under the map $ Y \xlongrightarrow{p_-} Y_{-}$, it follows that under the blow-up map 
\begin{equation*}
\mathbb{P}_{Y}(\mathcal{L}_{-}\oplus \mathcal{L}_{+}) \xlongrightarrow{q_-} \mathbb{P}_{Y_{-}}(\mathcal{G}_{-}),
\end{equation*}

$ \mathcal{O}(1)$ pulls back to $ \mathcal{O}(1).$ Hence, 
\begin{equation*}
(q_-)^{\ast}(\pi_{-})^{\ast}\mathcal{O}_{X}(1)= (q_-)^{\ast}\mathcal{O}_{\mathbb{P}(\mathcal{G}_{-})}(1), 
\end{equation*}

which implies, $ (\pi_{-})^{\ast}\mathcal{O}_{X}(1)= \mathcal{O}_{\mathbb{P}(\mathcal{G}_{-})}(1)$.

  Since $F$ fits into an exact sequence 

\begin{equation*}
    0 \longrightarrow \mathcal{O}_{Y_{-}}\xlongrightarrow{s}\mathcal{G}_{-} \longrightarrow F \longrightarrow 0,     
\end{equation*}

the divisor $D:= \mathbb{P}_{Y}(F)$ is an effective divisor in $ \mathbb{P}(\mathcal{G}_{-})$ in the linear system $ | \mathcal{O}_{\mathbb{P}(\mathcal{G}_{-})}(1)|.$ Note that 
\begin{equation*}
H^{0}(Y_{-}, \mathcal{G}_{-}) =  H^{0}(\mathbb{P}(\mathcal{G}_{-}), \mathcal{O}_{\mathbb{P}(\mathcal{G}_-)}(1)) = H^{0}(X,\mathcal{O}_{X}(1))=V. 
\end{equation*}
A non zero section $s \in H^{0}(Y_{-}, \mathcal{G}_{-})$ corresponding to a hyperplane $H_{s}$ in $ \mathbb{P}(V)$ and $ D= \pi_{-}^{-1}(H_{s} \cap X)$. Thus, $H_{s} \cap X=\pi_{-}(D)$ is irreducible.

Write $s= (s_{-},s_{+}),$ where $ s_{-}\in V_{-}= H^{0}(Y_{-},L_{-}),$ $ s_{+} \in V_{+}= H^{0}(Y_{+},L_{+})= H^{0}(Y_{-}, \mathcal{E}_{-}).$  If $s_{+}=0,$ then since $s$ is nowhere vanishing, we must have $s_-$ is nowhere vanishing, implying that $L_{-}$ is trivial, which is a contradiction. Hence, $s_{+}\neq 0.$ Therefore,
$$ Y_{+} \cap H_{s}= \big \{ y_{+}\in Y_{+} \hspace{2pt} | \hspace{2pt} s_{+}(y_{+}) = 0  \big \} \neq  Y_{+}.$$ So, $ Y_{+} \not \subset H_{s}. $

Since $\pi_{-}$ is an isomorphism away from $Y_{+},$ and $\pi_{-}^{-1}(H_{s} \cap X)= D = \mathbb{P}_{Y_{-}}(F)$ is smooth, $ H_{s} \cap X$ is smooth away from $H_{s} \cap Y_{+}$. Moreover, codim$ (H_{s} \cap Y_{+}, H_{s}\cap X) = $
 codim$(Y_{+},X)  \geq 2$, since  $Y_{+} \not \subset H_{s}$. Thus, $ H_{s} \cap X$ is regular in codimension 1. As, $ H_{s}\cap X $ is a hyperplane section of $X$, it is generically reduced, so $ H_{s}\cap X$ is reduced, and is $S_{2}$. By Serre's criterion of normality, $ H_{s}\cap X$ is a normal variety. Let $ \pi_{1}= {\pi_{-}}{\mid_{D}}:D \longrightarrow H_{s}\cap X$. We have that $ \pi_{1}$ is an isomorphism away from $ H_{s}\cap Y_{+},$ and over $ H_{s}\cap Y_{+}$, $ \pi_{1}$ is a $ \mathbb{P}^{k_+}$- bundle, where $ k_{+}= \textrm{dim} Y - \textrm{dim}Y_{+}.$

Therefore, $\pi_{1}$ is the contraction of an extremal ray $R$ of $D$. The dimension of Ex$(\pi_{1})$ is given by Ex$(\pi_{1}) = \textrm{dim}(H_{s} \cap Y_{+}) + k_{+}
= \textrm{dim} \hspace{2pt}Y_{+} -1+k_{+} = \textrm{dim} \hspace{2pt} Y-1  =  \textrm{dim} \hspace{2pt}\mathbb{P}_{Y_{-}}(F)-1 = \textrm{dim} \hspace{2pt}D -1.$
So, $ \pi_{1}$ is divisorial contraction. Let $ y_{1},y_{2}\in Y_{+}$ be such that $ y_{1}\in H_{s}$,  $ y_{2}\neq H_{s}.$ Let $C_{i}$ be a line in $ \pi_{-}^{-1}(y_{i}) \cong \mathbb{P}^{k_{+}}.$ The length of $R$ is given by
$$ \ell(R) = -K_{D}\cdot C_{1}= -(K_{\mathbb{P}(\mathcal{G}_{-})}+ D)\cdot C_{1} \hspace{20pt} \textrm{(by Adjunction Formula)} $$
$$\hspace{28pt} = (K_{\mathbb{P}(\mathcal{G}_{-})}+D) \cdot C_{2} \hspace{30pt} (\textrm{as}\hspace{3pt} C_{1} \equiv C_{2})$$
$$\hspace{35pt} = -K_{\mathbb{P}(\mathcal{G_{-}})} \cdot C_{2} \hspace{30pt} (\textrm{since}\hspace{3pt} C_{2} \cap D  =\varnothing )$$
$$\hspace{10pt} =k_{+}
\textrm{ (as} \hspace{2pt}\mathbb{P}(\mathcal{G}_{-}) \xlongrightarrow{\pi_{-}} \textrm{X is smooth blow-up of a smooth variety}$$ 
$$ \hspace{100pt}\textrm{along a smooth subvariety of codimension} \hspace{3pt}k_{+}+1 ). $$
By {\cite [Theorem 5.2]{AO}}, $ \pi_{1}$ is blow-up of a smooth variety along a smooth subvariety. Thus, $ H_{s}\cap X$ and $H_{s} \cap Y_{+}$ are smooth, and $ \mathbb{P}_{Y_{-}}(F)\cong \textnormal{Bl}_{H_{s}\cap Y_{+}}(H_{s}\cap X).$ 
\end{proof}

\hspace{10pt} Of course, by switching $+$ and $-$ in Theorem \ref{theorem:theorem drum}, we obtain a symmetric and valid theorem. Also, note that if rk$(\mathcal{E}_-)\geq\, $dim$Y_-$, then by \cite[Theorem 2.8]{Ot}, a general section of $L_-\oplus\mathcal{E}_-$ is nowhere vanishing. Therefore, Theorem \ref{theorem:theorem drum}$(ii)$, indeed provides examples of projective bundles with smooth blow-up structures.
 
\hspace{10pt}To provide explicit applications of Theorem \ref{theorem:theorem drum}, we need to describe the drum $X$. In {\cite [Propositions 1.8,1.9,1.10,1.12]{Pa}}, the drum is computed for certain  rational homogenous varieties $Y$ that have two projective bundle structures. See also {\cite [Examples 2.14,2.16]{LF}}. For these $Y$, $Y_{\pm}$ can be obtained from the first table in \cite[Section $4$]{CRC}. This leads to the following corollary:  

\begin{corollary}\label{theorem:theorem grassmanian}
 \begin{enumerate}
 \item {For 1 $\leq r \leq m-1$ } we have: 
 \item [{1.1}] The blow-up of \emph{Gr}$(r+1, m+2)$ along \emph{Gr}$(r, m+1)$ is a projective bundle over \emph{Gr}$(r+1, m+1)$.
 \item [{1.2}] The blow-up of \emph{Gr}$(r+1, m+2)$ along \emph{Gr}$(r+1, m+1)$ is a projective bundle over \emph{Gr}$(r, m+1)$.
\item {For $m\geq3$, blow-up of \emph{OG}$(m, 2m+1)$ along \emph{OG}$_{+}(m, 2m)(\cong$ \emph{OG}$(m-1, 2m-1)$) is a projective bundle over \emph{OG}$(m-1, 2m-1)$}. 
\item Let $m\geq 2$, $1\leq r \leq m-1$, $w$ a skew-form on $k^{2m+1}$ of maximal rank. We have:
\item [{3.1}] The blow-up of \emph{Gr}$_{w}(r+1, 2m+1)$ along \emph{SG}$(r, 2m)$ is a projective bundle over \emph{SG}$(r+1, 2m)$.
\item[{3.2}] The blow-up of \emph{Gr}$_{w}(r+1, 2m+1)$ along \emph{SG}$(r+1, 2m)$ is a projective bundle over \emph{SG}$(r+1, 2m)$.
\end{enumerate}
 
 \end{corollary}
 Here the embeddings between Grassmanians are explained in $\S 2$.
 
\hspace{10pt} When $Y_{-}$ is a projective space and $Y$ rational homogeneous, we can also identify the vector bundle $\mathcal{E}_{-}$ on Y$_{-}$.
\begin{corollary}
    
\label{theorem:theorem r}\begin{enumerate}
\item  The following statements hold: 

$i)$ The projective bundle $\mathbb{P}_{\mathbb{P}^n}(\mathcal{O}_{\mathbb{P}^n}(1) \oplus T_{\mathbb{P}^n}(-1))$ is the blow- up of a smooth quadric in $\mathbb{P}^{2n+1}$ along a linear subvariety of dimension $n$.
    
 $ii)$ If 
\begin{equation}
0\longrightarrow \mathcal{O}_{\mathbb{P}^n}\xlongrightarrow {s}{\mathcal{O}_{\mathbb{P}^n}(1) \oplus T_{\mathbb{P}^n}(-1)}\longrightarrow F\longrightarrow 0
\end{equation}
 is an short exact sequence of vector bundles, where $s$ is a nowhere vanishing section, then $\mathbb{P}_{\mathbb{P}^n}(F)$ is the blow up of the smooth quadric in $\mathbb{P}^{2n}$ along a linear subvariety of dimension $n-1$.
 \item 
 Let $ n\geq 3$. The following statements hold.
     
(i) $\mathbb{P}_{\mathbb{P}^n}(\mathcal{O}_{\mathbb{P}^n}(1) \oplus \Omega_{\mathbb{P}^n}(2)) $ is the blow-up of $\emph{Gr}(2,n+2)$ along $\emph{Gr}(2,n+1)$.    

(ii) If $s$ is a nowhere vanishing section of $ \mathcal{O}_{\mathbb{P}^n}(1)\oplus \Omega_{\mathbb{P}^n}(2)$ and the vector bundle $E$ is defined by the exact sequence 

$$ 0\longrightarrow \mathcal{O}_{\mathbb{P}^n}(1) \longrightarrow \mathcal{O}_{\mathbb{P}^n}(1) \oplus \Omega_{\mathbb{P}^n}(2) \longrightarrow E \longrightarrow 0  ,$$ then $ \mathbb{P}_{\mathbb{P}^n}(E)$ is the blow-up of the $ H \cap \emph{Gr}(2,n+2)$ along $ H \cap \emph{Gr}(2,n+1)$, where $ H \cap \emph{Gr}(2,n+2)$ is a smooth hyperplane section of $ \emph{Gr}(2,n+2)$ under the Plücker embedding, such that $ H\cap \emph{Gr}(2,n+1)$ is also smooth. For example, this gives smooth blow-up structure on $\mathbb{P}_{\mathbb{P}^{n}}(\mathcal{O}_{\mathbb{P}^{n}}(1)\oplus \mathcal{N}_{\mathbb{P}^n}(1))$, where $n$ is odd and $\mathcal{N}_{\mathbb{P}^n}$ is a null-correlation bundle on $\mathbb{P}^n.$
 
\end{enumerate}
\end{corollary}

\begin{proof}
\begin{enumerate}
\item {As in \cite[Example 2.16]{LF}}, $Y = \mathbb{P}_{\mathbb{P}^{n}}(T_{\mathbb{P}^{n}})$ has two projective bundle structures over $\mathbb{P}^{n}$, and the drum constructed upon it is $Q_{2n}$. Since $Y\cong \mathbb{P}_{\mathbb{P}^n}(\mathcal{E}_-)$ over $Y_-=\mathbb{P}^n$, we have $\mathcal{E}_{-} = T_{\mathbb{P}^{n}}(j)$ for some $j \in \mathbb{Z}$. We also have:
\begin{equation*}
h^{0}(\mathbb{P}^{n}, T_{\mathbb{P}^{n}}(j)) = h^{0}(Y_{-}, \mathcal{E}_{-}) = h^{0}(Y_{+}, \mathcal{L}_{+}) = n+1.    
\end{equation*}
By Euler's exact sequence, we have the short exact sequence 
\begin{equation*}
 0\longrightarrow \mathcal{O}_{\mathbb{P}^{n}}(i)\longrightarrow \mathcal{O}_{\mathbb{P}^{n}}(i+1)^{n+1}\longrightarrow T_{\mathbb{P}^{n}}(i) \longrightarrow 0. 
\end{equation*}
From this one easily get:
\begin{equation*}
 h^{0}(\mathbb{P}^{n}, T_{\mathbb{P}^{n}}(i)) = 0 \quad \textnormal{if}\quad i<-1,\qquad h^{0}(\mathbb{P}^{n}, T_{\mathbb{P}^{n}}(-1)) = n+1.  
\end{equation*}
For $i\geq 0$, we have $\mathcal{O}_{\mathbb{P}^{n}}(i+1)$ very ample and $T_{\mathbb{P}^{n}}(-1)$ is globally generated. Thus
\begin{equation*}
h^{0}(\mathbb{P}^{n}, T_{\mathbb{P}^{n}}(i)) = h^{0}(\mathbb{P}^{n}, T_{\mathbb{P}^{n}}(-1)\otimes \mathcal{O}_{\mathbb{P}^{n}}(i+1)) > h^{0}(\mathbb{P}^{n}, T_{\mathbb{P}^{n}}(-1)) = n+1.   
\end{equation*}
So, $j = -1$. Now Theorem \ref{theorem:theorem drum} gives (1). 
\\
Note that $Y_{-}\hookrightarrow X$ is a linear subspace here as $\mathcal{O}_{X}(1)_{{|}_{{Y_{-}}}} = \mathcal{O}_{{Y}_{-}}(1) = \mathcal{O}_{\mathbb{P}^{n}}(1)$. 
\item{As in \cite[Examples 3.11]{BSV}, let Y = $\mathbb{P}_{\mathbb{P}^{n}}(\Omega_{\mathbb{P}^{n}})$ be the partial flag variety with the two projective bundle structure over $\mathbb{P}^{n}$, Gr$(2, n+1)$, respectively.} We have $Y = A_{n}(1, 2)$ in the notation of \cite[Remark 3.3]{OCR}. By \cite[Proposition 1.9]{Pa}, the drum $X$ constructed upon it is Gr$(2, n+2)$. Thus, $\mathcal{E}_{-} = \Omega_{\mathbb{P}^{n}}(j)$ for some $j \in \mathbb{Z}$. Euler's exact sequence gives a short exact sequence:

\begin{equation*}
0\longrightarrow \Omega(i)\longrightarrow\mathcal{O}(i-1)^{n+1}\longrightarrow \mathcal{O}(i)\longrightarrow 0.
\end{equation*}
By Bott vanishing (\cite{OSS}, page 8),
$h^{1}(\mathbb{P}^{n}, \Omega(i)) = 0 \text{ for } i\geq 1$.    
Hence, 
\begin{equation*}
 h^{0}(\mathbb{P}^{n},\Omega(i)) = 0 \quad \textnormal{if}\quad i\leq 1,\qquad h^{0}(\mathbb{P}^{n}, \Omega(2)) = \frac{n^{2}+n}{2},  
\end{equation*}

For $i>2$, by a similar argument as in the proof of part (1) , and using the fact that $\Omega(2)$ is globally generated by Theorem \ref{theorem:theorem cl}, we get 

\begin{equation*}
 h^{0}(\mathbb{P}^{n}, \Omega(i))> \frac{n^{2}+n}{2}.   
\end{equation*}
Since 
\begin{equation*}
h^{0}(\mathbb{P}^{n}, \Omega(j)) = h^{0}(\mathbb{P}^{n}, \mathcal{E}_{-}) = h^{0}(Y_{+}, L_{+}) = \frac{n^{2}+n}{2}, 
\end{equation*}
we get $j = 2$. Now the rest of the proof follows Theorem \ref{theorem:theorem drum}.
\end{enumerate}    
\end{proof}
\begin{remark}
Note that $1.ii)$ of Corollary \ref{theorem:theorem r}, generalizes Theorem 1 in \textnormal{\cite{Va}} to higher dimensions.
\end{remark}
\begin{remark}
    
For $n=2$ Corollary \ref{theorem:theorem r}, $(ii)$ is still true, but this is essentially corollary $(i)$, as the perfect pairing $ \Omega_{\mathbb{P}^2} \times \Omega_{\mathbb{P}^2} \longrightarrow \Omega_{\mathbb{P}^{2}}^{2}= \mathcal{O}_{\mathbb{P}^2}(-3) $ shows that $ T_{\mathbb{P}^{2}} = \Omega_{\mathbb{P}^2}^{\ast}= \Omega_{\mathbb{P}^2}(3)$.  

\end{remark}

\section{Projective bundle as a blow-up of a singular variety along a linear subspace}

Consider the following setup. Let $d, n \geq 2$ be integers and define 
\begin{equation*}
 V= \{\underline{x}\in \mathbb{P}^{2n+1} \,|\, x_{0}^{d-1}x_{1}+ x_{2}^{d-1}x_{3}+\ldots+ x_{2n}^{d-1}x_{2n+1}=0\},
\end{equation*}
\begin{equation*}
 L_{0} = \{ \underline{x} \in \mathbb{P}^{2n+1} \medspace | \medspace x_{0} = x_{2} = \ldots = x_{2n}=0 \}.
 \end{equation*}

\hspace{10pt}Here, $L_0$ is an $n-$dimensional linear subspace contained in $V$. The variety $V$ is a smooth quadric if $ d=2 $, and \textnormal{Sing(V)} = $L_{0}$, if $ d\geq 3$.

\hspace{10pt} For $ r\geq 0$, identify  $\mathbb{P}^{2n+1}$ as the linear subvariety of $\mathbb{P}^{2n+1+r}$ consisting of points with the last $r$ homogeneous coordinates $0$. Let $M$ be the ($r-1$)-dimensional linear subspace of $\mathbb{P}^{2n+1+r}$ consisting of points with the first $2n+1$ homogeneous coordinates equal to $0$. For a subvariety $W$ of $\mathbb{P}^{2n+1}$, denote by $C(W)$ the cone over $W$ with vertex $M$. Therefore, 
$ C(L_{0})$ is a linear subspace of dimension $n+r$ contained in  $C(V)$. 

In what follows, $(\alpha_0,\alpha_1,...,\alpha_n)$ are the homogeneous coordinate of $\mathbb{P}^n$, $(x_0,x_1,...,x_{2n+1+r})$ is the homogeneous coordinates of $\mathbb{P}^{2n+1+r}$.

\hspace{10pt}Consider the  surjection of vector bundles:
\begin{equation*}
\mathcal{O}_{\mathbb{P}^{n}}^{n+1}\xlongrightarrow[(\alpha_{0}^{d-1},\ldots,\alpha_{n}^{d-1})]{\psi} \mathcal{O}_{\mathbb{P}^{n}}(d-1).
\end{equation*}

\hspace{10pt}Let $ F= \textnormal{Ker} \medspace \psi$, which is a rank $n$ vector bundle over $ \mathbb{P}^n$. Hence, we have the following exact sequence: 

\begin{equation}
0\longrightarrow F\longrightarrow\mathcal{O}_{\mathbb{P}^{n}}^{n+1}\xlongrightarrow{\psi} \mathcal{O}_{\mathbb{P}^{n}}(d-1)\longrightarrow 0.
\end{equation}

\hspace{10pt}Taking the dual of the eq.(8), we have the following sequence:
\begin{equation}
0 \longrightarrow \mathcal{O}_{\mathbb{P}^n}(-d+1) \longrightarrow \mathcal{O}_{\mathbb{P}^n}^{n+1} \longrightarrow F^{\ast} \longrightarrow 0.
\end{equation}
\hspace{10pt}For example, if we set $d=2$, in eq.(9), then Euler's exact sequence shows that $F^* = T_{\mathbb{P}^n}(-1)$.

\hspace{10pt}We have an isomorphism 

\begin{equation*}
k^{2n+2+r}\xlongrightarrow{\phi} k^{r}\oplus k^{n+1}\oplus H^{0}(\mathbb{P}^{n}, \mathcal{O}_{\mathbb{P}^{n}}^{r+n+1}\oplus \mathcal{O}_{\mathbb{P}^{n}}(1))  
\end{equation*}
\begin{equation*}
(x_{0},\ldots,x_{2n+1+r})\longmapsto ((x_{2n+2},\ldots,x_{2n+1+r}),(x_{1},\ldots,x_{2n+1}),\sum_{i=0}^{n}x_{2i}\alpha_{i}).     
\end{equation*}
This yields 
\begin{equation*}
p: \mathbb{P}_{\mathbb{P}^{n}}(\mathcal{O}_{\mathbb{P}^{n}}^{r+n+1}\oplus\mathcal{O}_{\mathbb{P}^{n}}(1))\longrightarrow \mathbb{P}(H^{0}(\mathbb{P}^{n}, \mathcal{O}_{\mathbb{P}^{n}}^{r+n+1}\oplus \mathcal{O}_{\mathbb{P}^{n}}(1))),    
\end{equation*}
which is the blow-up along $C(L_{0})$.

For $ r\geq 0$  we have a surjection: 
\begin{equation}
 \mathcal{O}_{\mathbb{P}^n}^{r+n+1}\oplus \mathcal{O}_{\mathbb{P}^n}(1)\longrightarrow \mathcal{O}_{\mathbb{P}^n}^r\oplus F^{\ast}\oplus \mathcal{O}_{\mathbb{P}^n}(1)
\end{equation}
 We have the induced map $ \pi: \mathbb{P}_{\mathbb{P}^n}(\mathcal{O}_{\mathbb{P}^n}^r \oplus F^{\ast} \oplus \mathcal{O}_{\mathbb{P}^n}(1)) \hookrightarrow \mathbb{P}_{\mathbb{P}^n}(\mathcal{O}_{\mathbb{P}^n}^{r+n+1} \oplus \mathcal{O}_{\mathbb{P}^n}(1)) \xlongrightarrow{p} \mathbb{P}^{2n+1+r}$.

\hspace{10pt}Now we can prove the result.

\begin{theorem}\label{theorem:theorem conee}
The map $\pi$  is the blow-up of $C(V)$ along $C(L_{0})$.
\end{theorem}

\begin{proof}
Clearly, $\mathcal{O}_{\mathbb{P}^n}^{r+n+1} \oplus \mathcal{O}_{\mathbb{P}^n}(1) \xlongrightarrow {\psi} \mathcal{O}_{\mathbb{P}^n}^{r}\oplus F^{\ast} \oplus \mathcal{O}_{\mathbb{P}^n}(1)$ satisfies $ \psi({\mathcal{O}_{\mathbb{P}^n}(1)})= \mathcal{O}_{\mathbb{P}^n}(1) \neq 0 $. By Lemma \ref{theorem:theorem key}, it suffices to show $\text{Im}\,\pi  = C(V)$.

The description of the map $p$ gives:
$$ \text{Im}(\pi) = \{ [\underline{x}] \in \mathbb{P}^{2n+1+r} \medspace| \medspace  \exists\text{ } \underline{\alpha}\in \mathbb{P}^n,\, \textrm{such that} $$ 
$$ k^{r} \oplus F_{\underline{\alpha}}\oplus \mathcal{O}(-1)_{\underline{\alpha}}    \ni (( x_{2n+2},\ldots,x_{2n+1+r}),(x_{1},\ldots,x_{2n+1}),(x_{0},\ldots,x_{2n}))\}.$$
Therefore, 
\begin{equation*}
[\underline{x}]\in \text{Im}\,\pi\Leftrightarrow \text{there exist}\, [\underline{\alpha}] \in \mathbb{P}^n\, \textnormal{and}\, \lambda\in k \, \textnormal{such \,that}
\end{equation*}

\begin{equation*}
x_{2i}= \lambda \alpha_{i},\, \text{for}\medspace 0 \leq i \leq n, \,\,\,\sum_{i=0}^{n} \alpha_{2i}^{d-1} x_{2i+1}=0
\end{equation*}

\begin{equation*}
\Leftrightarrow \sum_{i=0}^{n} x_{2i}^{d-1}x_{2i+1} =0 \Leftrightarrow [\underline{x}] \in C(V).
\end{equation*}   
\end{proof}

Since $ H^{1}(\mathbb{P}^n, \mathcal{O}_{\mathbb{P}^n}(-d+1)) =0$ and $H^{0}( \mathbb{P}^n, \mathcal{O}_{\mathbb{P}^n}(-d+1))$, eq.(9) shows that $H^{0}(\mathbb{P}^n, F^{\ast}) = H^{0}(\mathbb{P}^n, \mathcal{O}_{\mathbb{P}^n}^{n+1}).$
Thus, via $\phi$, $H^{0}(\mathbb{P}^n, F^{\ast} \oplus \mathcal{O}_{\mathbb{{P}}^n}(1) )= H^{0}(\mathbb{P}^n, \mathcal{O}_{\mathbb{P}^n}^{n+1} \oplus \mathcal{O}_{\mathbb{{P}}^n}(1) )=k^{2(n+1)}$. For $\underline{x}\in k^{2n+1}$, let $s_{\underline{x}}$ be the corresponding section of $F^{\ast} \oplus \mathcal{O}_{\mathbb{{P}}^n}(1)$. Note that $ s_{\underline{x}}$ vanishes at $ \underline{\alpha} \in \mathbb{P}^n$ if and only if $\sum_{i=0}^{n} x_{2i} \alpha_{i}=0$
and $ (x_{1},x_{3},\ldots,x_{2n+1})\in\text{Span} \{(\alpha_{0}^{d-1}, \alpha_{1}^{d-1},\ldots, \alpha_{n}^{d-1})\}$.
Therefore, $s_{\underline{x}} $ is nowhere vanishing if and only if 
\begin{equation}
\sum_{i = 0}^{n} x_{2i}y_{2i+1}\neq 0 \quad \forall \, y_{1}, y_{3},\ldots,y_{2n+1}\in k, \, \text{where}\, y_{2i+1}^{d-1} = x_{2i+1}\,\text{for all}\, i. 
\end{equation} Eq.(11) implies that a general section of $F^{\ast} \oplus \mathcal{O}_{\mathbb{{P}}^n}(1)$ is nowhere vanishing. This also follows from \cite[Theorem 2.8]{Ot}, as $F^{\ast} \oplus \mathcal{O}_{\mathbb{{P}}^n}(1)$ is globally generated and has rank $n+1>n$.

\hspace{10pt} Now we want to study the projectivization of the quotient of $ \mathcal{O}_{\mathbb{P}^n}(1)\oplus F^{\ast}$ by a trivial subbundle. 

\begin{theorem}\label{theorem:theorem sing}
Let $\underline{x}\in k^{2n+2}\setminus0$ be such that $s=s_{\underline{x}}$ is a nowhere vanishing section of $ \mathcal{O}_{\mathbb{P}^n}(1)\oplus F^{\ast}$. Let $E$ be the rank $n$ vector bundle over $\mathbb{P}^n$ defined by the following exact sequence:
\begin{equation*}
0\longrightarrow \mathcal{O}_{\mathbb{P}^n}\xlongrightarrow{s} \mathcal{O}_{\mathbb{P}^n}(1)\oplus F^{\ast} \longrightarrow E \longrightarrow 0.
\end{equation*}
Thus, we can regard $ \mathbb{P}_{\mathbb{P}^n}(E)$ as a codimension $1$ subvariety of $\mathbb{P}_{\mathbb{P}^n}(\mathcal{O}_{\mathbb{P}^n}(1)\oplus F^{\ast})$.
  The hyperplane $H_{\underline{x}}:= \{\underline{y} \in \mathbb{P}^{2n+1} \medspace | \medspace \sum_{i=0}^{2n+1} y_{i} x_{i}=0 \} \not \supset L_{0} $ and $ \pi \mid_{\mathbb{P}_{\mathbb{P}^n}(E)}$ identifies $ \mathbb{P}_{\mathbb{P}^n}(E)$ as the blow-up of $ V \cap H_{\underline{x}}$ along $L_{0} \cap H_{\underline{x}}$.
\end{theorem}
\begin{proof}
Since $\underline{x}$ satisfies eq.(11), not all $x_{2i}$'s can be zero. Thus, $H_{\underline{x}}\not \supset L_{0}$.
By Theorem \ref{theorem:theorem conee}, $  \pi : \mathbb{P}_{\mathbb{P}^{n}}(\mathcal{O}_{\mathbb{P}^n}(1)\oplus F^{\ast}) \longrightarrow V$ is the blow up along $L_0$.
Since $ \mathbb{P}_{\mathbb{P}^n}(E) $ is a codimension 1 subvariety of $\mathbb{P}_{\mathbb{P}^n}(\mathcal{O}_{\mathbb{P}^n}(1)\oplus F^{\ast})$, $ \mathbb{P}_{\mathbb{P}^n}(E)$ is the strict transform of $ \pi(\mathbb{P}_{\mathbb{P}^n}(E))$ under $ \pi$. In other words, $ \pi \mid_{\mathbb{P}_{\mathbb{P}^n}(E)}$ identifies $ \mathbb{P}_{\mathbb{P}^n}(E)$ with the blow-up of $ \pi(\mathbb{P}_{\mathbb{P}^n}(E))$ along $ \pi (\mathbb{P}_{\mathbb{P}^n}(E))\cap L_{0}.$ It remains to show that $ \pi(\mathbb{P}_{\mathbb{P}^n}(E)) = H_{\underline{x}} \cap V$. 
The vector $\underline{x}$ induces a map $ \mathcal{O}_{\mathbb{P}^n} \xlongrightarrow{\underline{x}} \mathcal{O}_{\mathbb{P}^n}^{2n+2}$ such that $ s$ is the composition:
\begin{equation*}
\mathcal{O}_{\mathbb{P}^n}\xlongrightarrow{\underline{x}} \mathcal{O}_{\mathbb{P}^n}^{2n+2}\longrightarrow \mathcal{O}_{\mathbb{P}^n}(1)\oplus F^{\ast}.
\end{equation*}

There is an induced map $\textnormal{Coker}\medspace {\underline{x}} \longrightarrow E$ such that we have a pushout diagram:
 \begin{center}
\begin{tikzcd}
\mathcal{O}_{\mathbb{P}^n}^{2n+2} \arrow[r] \arrow[d,"\psi"]
& \mathcal{O}_{\mathbb{P}^n}(1)\oplus F^{\ast} \arrow[d, "\eta" ] \\
\textnormal{Coker}\, \underline{x} \arrow[r, ]
& |[, rotate=0]| E
\end{tikzcd}
\end{center}
So, $\mathbb{P}_{\mathbb{P}^n}(E)= \mathbb{P}_{\mathbb{P}^n}(\mathcal{O}_{\mathbb{P}^n}(1)\oplus F^{\ast}) \cap \mathbb{P}_{\mathbb{P}^n}(\textnormal {Coker}\, \underline{x}$) in $ \mathbb{P}_{\mathbb{P}^n}(\mathcal{O}_{\mathbb{P}^n}^{2n+2})= \mathbb{P}^n\times \mathbb{P}^{2n+1}$, that is,
$$ \mathbb{P}_{\mathbb{P}^n}(E)= \{(\underline{\alpha},\underline{y})\in \mathbb{P}^{n}\times\mathbb{P}^{2n+1}\,|\;(\underline{\alpha}, \underline{y}) \in \mathbb{P}_{\mathbb{P}^n}(\mathcal{O}_{\mathbb{P}^n}(1)\oplus F^{\ast}),\, \underline{y} \in H_{\underline{x}}\} = \pi^{-1}(H_{\underline{x}}). $$
Thus, $\pi(\mathbb{P}_{\mathbb{P}^n}(E))= H_{\underline{x}}\cap V$. This completes the proof.
\end{proof}
As a corollary, we get the generalization in higher dimension of Theorem 3 in \cite{Va}.
 \begin{remark} Note that Theorems \ref{theorem:theorem conee} (with $r=0$) and \ref{theorem:theorem sing} gives another proof of Theorem \ref{theorem:theorem r} $(1)$.
 \end{remark}
 
 \begin{corollary}\label{theorem:theorem shivam}
  Let $d, n \geq 2$ be integers, and let
\begin{equation*}
V= \{\underline{x}\in \mathbb{P}^{2n+1} \,|\, x_{0}^{d-1}x_{1}+ x_{2}^{d-1}x_{3}+\ldots+ x_{2n}^{d-1}x_{2n+1}=0\}
\end{equation*}
\begin{equation*}
 L_{0} = \{ \underline{x} \in \mathbb{P}^{2n+1} \medspace | \medspace x_{0} = x_{2} = \ldots = x_{2n}=0 \}.
 \end{equation*}
 
 There is a nonempty open set $U_{n,d}$ of $\mathbb{P}^{2n+1}$ such that, for each point $\sigma\in U_{n,d}$ we have a rank $n$ vector bundle $E_{n,d,\sigma}$ over $\mathbb{P}^n$, and a hyperplane $H_\sigma$ in $\mathbb{P}^{2n+1}$ does not contain $L_0.$ The linear system $|\mathcal{O}_{\mathbb{P}(E_{n,d,\sigma})}(1)|$ defines a morphism $\mathbb{P}_{\mathbb{P}^n}(E_{n,d,\sigma})\to \mathbb{P}^{2n+1}$, which identifies $\mathbb{P}_{\mathbb{P}^n}(E_{n,d,\sigma})$ as the blow-up of $V\cap H_\sigma$ along $L_0\cap H_\sigma$.
 \end{corollary}
 \begin{proof}
 We have seen $H^{0}(\mathbb{P}^n, \mathcal{O}_{\mathbb{{P}}^n}(1)\oplus F^{\ast}) =k^{2n+2}$. There exists a nonempty open set $U_{n,d}$ of $\mathbb{P}^{2n+1}=\mathbb{P}(H^{0}(\mathbb{P}^n, \mathcal{O}_{\mathbb{{P}}^n}(1)\oplus F^{\ast})^*)$, such that any $\sigma\in U_{n,d}$ corresponds to a nowhere vanishing section $s$ in $H^{0}(\mathbb{P}^n, \mathcal{O}_{\mathbb{{P}}^n}(1)\oplus F^{\ast})$. Define $E_{n,d,\sigma}=E$, $H_\sigma=H_{\underline{x}}$ in the notation of Theorem \ref{theorem:theorem sing}. Now Theorem \ref{theorem:theorem sing} implies the corollary.
 \end{proof}
\vspace{10pt}

\begin{remark}\label{remark:remark cone}
Similarly as in Theorem \ref{theorem:theorem conee}, for all $ r\geq 1$, $ \mathbb{P}_{\mathbb{P}^n}( \mathcal{O}_{\mathbb{P}^n}^{r} \oplus E)$ is the blow-up of the singular quadric C$( Q) \hookrightarrow \mathbb{P}^{2n+r}$ along the linear subvariety $ \mathbb{P}^{n+r-1}\cong C(L) \hookrightarrow C(Q)$.
\end{remark}
\hspace{10pt}Now we can describe the blow-up structure on the projectivization of vector bundles as in $(iii)$ of Theorem \ref{theorem:theorem cl}.
\begin{theorem}\label{theorem:theorem 4.10}
Let $E$ be as in $(3)$ of Theorem \textnormal{\ref{theorem:theorem cl}}. Then $\mathbb{P}_{\mathbb{P}^n}(E)$ is the blow-up of a quadric in $\mathbb{P}^{n+r}$ along a linear subvariety of dimension $r-1$. The quadric is always smooth if $r=n$, and always singular if $r\geq n+2$.
\end{theorem}
\begin{proof}
For $r=n$, Theorem \ref{theorem:theorem r} $(ii)$ completes the proof. So assume $r-n\geq1.$\\
Let $ \tilde{s}$: $\mathcal{O}_{\mathbb{P}^{n}}\xlongrightarrow{(\underline{\lambda}, s)} \mathcal{O}_{\mathbb{P}^n}^{r-n}\oplus \mathcal{O}_{\mathbb{P}^n}(1)\oplus T_{\mathbb{P}^n}(-1)$ be a nowhere vanishing section, where $\underline{\lambda}\in k^r$, considered as a section of $\mathcal{O}_{\mathbb{P}^n}^{r}$, and $s$ is a section of $\mathcal{O}_{\mathbb{P}^n}(1)\oplus T_{\mathbb{P}^n}(-1)$. We consider two cases:

Case(1): $ \underline{\lambda} =0 .$

Then $s$ is nowhere vanishing, so $E = \textnormal{Coker}\medspace \tilde{s} \cong \mathcal{O}_{\mathbb{P}^{n}}^{r-n}\oplus\textnormal{Coker} \,s$. So, $ \mathbb{P}_{\mathbb{P}^n}(E)$ is the blow-up of a singular quadric in $ \mathbb{P}^{n+r}$ along a linear subvariety of dimension  $ r-1$, by Remark \ref{remark:remark cone}.

Case(2): $ \underline{\lambda} \neq 0$, say $ \lambda_{1} \neq 0$.

We have the following exact sequence:

$
\scriptstyle{0\longrightarrow \mathcal{O}_{\mathbb{P}^n}\xlongrightarrow{(\lambda_{1}, s_{1}) = \tilde{s}}\mathcal{O}_{\mathbb{P}^n}^{r-n}\oplus\mathcal{O}_{\mathbb{P}^n}(1)\oplus T_{\mathbb{P}^{n}}(-1) =  \mathcal{O}_{\mathbb{P}^n}\oplus(\mathcal{O}_{\mathbb{P}^n}^{r-n-1}\oplus \mathcal{O}_{\mathbb{P}^n}(1) \oplus T_{\mathbb{P}^n}(-1))\longrightarrow
\mathcal{O}_{\mathbb{P}^n}^{r-n-1}\oplus \mathcal{O}_{\mathbb{P}^n}(1)  \oplus T_{\mathbb{P}^n}(-1)  \longrightarrow 0}.$

\hspace{300pt}{$\scriptstyle (\beta, t)\mapsto t-\frac{\beta}{\lambda_{1}}s_{1}$}

So, $E_{1}= \textnormal{Coker} (\tilde{s})= \mathcal{O}_{\mathbb{P}^n}^{r-n-1}\oplus \mathcal{O}_{\mathbb{P}^n}(1)\oplus T_{\mathbb{P}^n} (-1).$ Therefore, $ \mathbb{P}_{\mathbb{P}^n}(E_{1})$ is the blow-up of a quadric $ Q \hookrightarrow \mathbb{P}^{n+r}$ along a linear subvariety of dimension $ r-1$. The quadric $Q$ is smooth if and only if $ r=n+1$. 
\end{proof}

\section{Projective bundles over projective plane with birational contractions to projective spaces}
We start with a linear algebraic lemma. For positive integers $a,b$, denote the space of $a\times b$ matrices with entries in $k$ by $M_{a\times b}(k).$
\begin{lemma} \label{theorem:linalg}
Let $s\geq1$ be an integer. Let $A,B \in M_{(s+2)\times 3}(k)  $, such that \emph{rk}($uA+vB)=3,\, \text{for all}\; 0\neq(u,v)\in k\times k$,\, \text{and}\, \emph{rk}$ [A:B]= r+2$. Then $ 2\leq r \leq 4$, and there exist $P\in \emph{GL}(r+2,k)$, $ Q \in \emph{GL}(3, k)$ such that $ PAQ = [e_{0}: e_{1}: e_{2}]$ and $PBQ=[e_{r-1}: e_{r}: e_{r+1}]$. In other words, there is a basis $ \{v_{0},v_{1},v_{2} \}$ of $ k^3$ and a basis $ \{ w_{0},w_{1},\dots,w_{s+1}\}$ of $ k^{s+2}$  such that $ A v_{i}=w_{i}$ and $ Bv_{i}=w_{r-1+i}$, for all $0\leq i\leq 2$.  
\end{lemma}
     
\begin{proof}
By a change of basis of $k^{s+2}$, we can assume that Im $A+$ Im $B=k^{r+2}\subset k^{s+2}.$ It suffices to find bases of $k^3$ and $k^{r+2}$ so that the condition in the lemma is satisfied. We can extend the basis of $k^{r+2}$ to a basis of $k^{s+2}$, and the condition in the lemma remains satisfied. Thus, without loss of generality, we can assume $s=r.$

Case 1: $r=4$

In this case, $[A:B]$ is a nonsingular $ 6\times 6$ matrix. So $ P= [A:B]^{-1}$ does the job. 

Case 2: $r = 3$

Let Im $A$, Im $B$  be the images of $ k^{3} \xlongrightarrow{A} k^{5}$, $ k^{3} \xlongrightarrow{B} k^{5}$. Since rk $[A:B]= r+2$, we have Im $A + $ Im $B= k^5$. Also, putting $(u,v)=(1,0)$ and $(0,1)$ in the hypothesis, we see that $A$ and $B$ have rank 3, so Im $A$ $\cap$ Im $B$ has dimension 1. Let $ w_{2} \in \textnormal{Im}\,{A}\, \cap\, \textnormal{Im}\,{B} $ be nonzero. Let $ v_{0}, v_{2} \in k^{3}$ be such that $ Av_{2}=w_{2} = Bv_{0}.$ If $ v_{0},v_{2}$ are linearly dependent, say $v_0=av_2$, then $(A-aB)v_2=0$, contradicting our hypothesis that rk($A-aB) =3$. So $ v_{0},v_{2}$ are linearly independent. Let $ v_{1}\in k^{3}$ be such that \{$ v_{0},v_{1}, v_{2}\}$ is  a basis of $k^{3}$. Let $w_{0} = Av_{0}, w_{1} = Av_{1}, w_{3} = Bv_{1}, w_{4} = Bv_{2}$. Then Span$\{w_{0},\ldots,w_{4}\}$ contains Span$\{w_{0}, w_{1}, w_{2}\} = \textrm{Im}\, A,$ and also contains Span$\{w_{2}, w_{3}, w_{4}\} = \textrm{Im}\, B$. Since Im $A$ + Im $B$ = $k^{5}$, $\{w_{0},\ldots,w_{4}\}$ is a basis of $k^{5}$.  

Case 3: $r=2$.

We have $W = \textrm{Im}\,A\cap \textrm{Im}\,B$ has dim$\,2$. Suppose $A^{-1}W = B^{-1}W = S\subseteq V$, then we have  dim $S = \textrm{dim}\,W = 2$ and 

$$A, B: S\longrightarrow W$$ 
are invertible linear transformations. So there exists $\lambda \in k$ such that linear transformation $A+\lambda B: S\longrightarrow W$ is not invertible. Therefore, the matrix $A+\lambda B$ can't have rank 3, as it has a non-empty kernel. Hence, $A^{-1}W \neq B^{-1}W$ and  $A^{-1}W$ and $B^{-1}W$ are distinct dimension 2 subspaces of $k^{3}$, which implies that dim$\,(A^{-1}W\cap B^{-1}W) = 1$,  and $\,(A^{-1}W + B^{-1}W) = k^{3}$. Let $v_{1}\in A^{-1}W\cap B^{-1}W$ be nonzero and $w_{1} = Av_{1}, w_{2} = Bv_{1}$. Note that $w_{1}, w_{2}\in W$ are linearly independent, because $(uA+vB)(v_1)\neq0,\, \text{for all}\; 0\neq(u,v)\in k\times k$. So, $w_{1}, w_{2}$ is a basis of $W$. Therefore, there are $v_{0}, v_{2}\in k^{3}$ with $Av_{2} = w_{2}, Bv_{0} = w_{1}$. Let $w_{0} = Av_{0}, w_{3} = Bv_{2}$. 

Claim: The sets $\{v_{0}, v_{1}, v_{2}\}$ and $\{w_{0},\ldots,w_{3}\}$ are bases of $k^{3}$ and $k^{4}$, respectively.

Proof of Claim : Clearly, $B(\textrm{Span}\{v_{0}, v_{1}\}) = \textrm\{w_{1}, w_{2}\} = W$. So, Span$\{v_{1}, v_{2}\} = A^{-1}W$. Hence, Span$\{v_{0}, v_{1}, v_{2}\}$ contains $A^{-1}W$, $B^{-1}W$, so Span$\{v_{0}, v_{1}, v_{2}\} = k^{3}.$ Thus, $\{v_{0}, v_{1}, v_{2}\}$ is a basis of $V$. Span$\{w_{0},\ldots,w_{3}\}$ contains both Im $A$ and Im $B$, so Span$\{w_{0},\ldots,w_{3}\}=k^4$. Hence, $\{w_{0},\ldots,w_{3}\}$ is a basis of $k^{4}$. Therefore, we are done.

 Case 4: $r=1$

 There is $P\in \textnormal{GL}(3,k)$ such that $A=BP,$ as $A$ and $B$ have the same column space. If $\lambda$ is an eigenvalue of $P$, $A-\lambda B=B(P-\lambda I)$ does not have rank $3$, a contradiction.

 So $2\leq r\leq 4$.
\end{proof}
\hspace{10pt}Now we describe the extremal contractions of the projectivization of the vector bundle in $(v)$ of Theorem \ref{theorem:theorem cl}.

\begin{theorem}\label{theorem:theorem nav}
 Let $E, r$ be as in $(5)$ of Theorem \textnormal{\ref{theorem:theorem cl}}, and let $t$ be the smallest integer such that $E$ has a trivial bundle of rank $r-t$ as a direct summand.  If $n\geq 3$, the other contraction of $\mathbb{P}_{\mathbb{P}^n}(E)$ is of fibre type onto $\mathbb{P}^{r+1}$ with general fibre $\mathbb{P}^{n-2}$. If $n=2$, then $2\leq t\leq 4,$ and the other contraction of $\mathbb{P}_{\mathbb{P}^n}(E)$ is birational onto $\mathbb{P}^{r+1}$. Up to an automorphism of $\mathbb{P}^{r+1}$, the fundamental locus is the cone over $V$, where \begin{equation*}
V = \begin{cases}
\textnormal{twisted cubic}, & \quad \textnormal{if}\quad t = 2, \\
Z(X_{0}X_{3} - X_{1}X_{2}, X_{1}X_{4} - X_{2}X_{3}, X_{0}X_{4} - X_{2}^{2})\hookrightarrow \mathbb{P}^{4}, & \quad\textnormal{if}\quad t = 3, \\
\textnormal{image of Segre embedding}\, \mathbb{P}^{1}\times\mathbb{P}^{2}\hookrightarrow\mathbb{P}^{5}, & \quad \textnormal{if}\quad t = 4.
\end{cases}
\end{equation*} 
$\mathbb{P}_{\mathbb{P}^n}(E)$ is a smooth blow-up if and only if $n=2, r=t$.
\end{theorem}
\begin{proof}
Let $r\geq n \geq 2$ be integers. Let the vector bundle $E$ be as in $(5)$ of Theorem \ref{theorem:theorem cl}:
\begin{equation*}
0\longrightarrow\mathcal{O}_{\mathbb{P}^n}(-1)^{2}\xlongrightarrow{s}\mathcal{O}_{\mathbb{P}^n}^{\,r+2}\longrightarrow E\longrightarrow 0.     
\end{equation*}
After taking the dual, $s$ corresponds to a surjection $\mathcal{O}_{\mathbb{P}^n}^{r+2}\xlongrightarrow{s^{*}}\mathcal{O}_{\mathbb{P}^n}(1)^{2}$. Let $\underline{\alpha} = [\alpha_{0}: \alpha_{1}:\ldots: \alpha_{n}]$ be the homogeneous coordinates in $\mathbb{P}^{n}$. There are matrices $A, B\in M_{(r+2)\times (n+1)}(k)$ such that $s^{*}(\underline{x}) = (\underline{x}^{t}A\underline{\alpha}, \underline{x}^{t}B\underline{\alpha})$. Then the following are equivalent:

\begin{enumerate}
\item [(i)] s$^{*}$ is surjective.
\item [(ii)] For all $\underline{\alpha}\neq 0, A\underline{\alpha}, B\underline{\alpha}$ are linearly independent.
\item [(iii)] For all  $0 \neq (u, v)\in k^{2}$, rk$(uA+vB) = n+1$.
\end{enumerate}

Let $f: \mathbb{P}_{\mathbb{P}^n}(E)\longrightarrow\mathbb{P}^{n}$ be the projection. Then map $\pi: \mathbb{P}_{\mathbb{P}^n}(E)\hookrightarrow\mathbb{P}_{\mathbb{P}^n}(\mathcal{O}_{\mathbb{P}^n}^{r+2})\longrightarrow\mathbb{P}^{r+1}$ has the fibers  
\begin{equation} 
\pi^{-1}([\underline{x}])\cong f\pi^{-1}([\underline{x}]) = \{\underline{\alpha}\in\mathbb{P}^{n}\big{|}\underline{x}^{t}A\underline{\alpha} = \underline{x}^{t}B\underline{\alpha} = 0\}
\end{equation}
$$
\cong 
\begin{cases}
\mathbb{P}^{n-2} , & \quad \textnormal{if}\quad \textrm{rk}\begin{pmatrix}
  \underline{x}^{t}A \\
  \underline{x}^{t}B \\
\end{pmatrix} = 2, \\
\mathbb{P}^{n-1}, & \quad\textnormal{if}\quad \textrm{rk}\begin{pmatrix}
  \underline{x}^{t}A \\
  \underline{x}^{t}B \\
\end{pmatrix}  = 1, \\
\mathbb{P}^{n} , & \quad \textnormal{if}\quad \underline{x}^{t}A=\underline{x}^{t}B=0.

\end{cases}
$$

From now on let us assume  $n=2$, so dim$\,\mathbb{P}_{\mathbb{P}^n}(E) = r+1 = \textrm{dim}\,\mathbb{P}^{\,r+1}$, and $\pi$ is surjective. Thus, $\pi$ is generically finite, and $\pi^{-1}([\underline{x}])\cong \mathbb{P}^{\, 0} = \{\textrm{pt}\}$ for general $[\underline{x}]\in\mathbb{P}^{\,2}$, i.e., $\pi$ is birational. The fundamental locus of $\pi$ is $S = \{\underline{X}\in \mathbb{P}^{\, r+1}\big{|}\:\textrm{rk}\begin{pmatrix}
  X^{t}A \\
  X^{t}B \\
\end{pmatrix}\} \leq 1$\}; an intersection of 3 quadrics, given by the determinants of the $2\times2$ minors of 
$
\begin{pmatrix}
  X^{t}A \\
  X^{t}B \\
\end{pmatrix}.
$ 
 Also, note that
 
 $E = \mathcal{O}_{\mathbb{P}^2}^{r-s} \oplus E_{1} $ for some bundle $E_1 \Longleftrightarrow   (v)  $ is the direct sum of $ 0 \longrightarrow 0 \longrightarrow \mathcal{O}_{\mathbb{P}^2} ^{r-s}   \longrightarrow \mathcal{O}_{\mathbb{P}^2}^{r-s} \longrightarrow 0$

\hspace{100pt} with an sequence $ 0 \longrightarrow \mathcal{O}_{\mathbb{P}^2}(-1)^{2} \longrightarrow \mathcal{O}_{\mathbb{P}^2}^{s+2} \longrightarrow E_{1} \longrightarrow 0 $ 

 \hspace{100pt} ( as the map $ \mathcal{O}_{\mathbb{P}^2} ^{r+2} \longrightarrow E $ in $ (v)$ is the evaluation map 

\hspace{120pt} $ H^{0}( \mathbb{P}^2 , E) \bigotimes \mathcal{O}_{\mathbb{P}^2}\longrightarrow E  $)

 \hspace{77pt} $ \Longleftrightarrow s^{\ast}$ is zero on a trivial bundle of $ \mathcal{O}_{\mathbb{P}^2}^{r+2}$  of rank  $ r-s$

  \hspace{77pt} $ \Longleftrightarrow $ dim $ \{ \underline{x} \in k^{r+2}  \hspace{2pt} |\hspace{2pt} x^{t}[A:B]=0 \} \geq r-s $

  \hspace{77pt} $ \Longleftrightarrow $ rk $ [A:B] \leq s+2 .$

  So $ t= \textrm{rk}\, [A:B]-2$.

By Lemma \ref{theorem:linalg}, $ 2\leq t \leq 4$ and there exists $P\in \textrm{GL($r+2, k$)}$, $Q\in \textrm{GL($3, k$)}$ such that $X^{t}PAQ = (X_{0}, X_{1}, X_{2})$ and  $ X^{t}PBQ = (X_{t-1}, X_{t}, X_{t+1})$. So, there is an automorphism $\phi$ of $\mathbb{P}^{2}$ such that $\phi^{*}E\cong \mathcal{O}_{\mathbb{P}^2}^{r-t}\oplus E_{1}.$ Here $E_{1}$ fits into an exact sequence
\begin{equation*}
0\longrightarrow \mathcal{O}_{\mathbb{P}^2}(-1)^{2}\xlongrightarrow{s_{1}}\mathcal{O}_{\mathbb{P}^2}^{t+2}\longrightarrow E_{1}\longrightarrow 0,
\end{equation*}
where $s_1$ is given by
\begin{equation*}
\mathcal{O}_{\mathbb{P}^2}^{t+2}\xlongrightarrow{s_{1}^{*}}\mathcal{O}_{\mathbb{P}^2}(1)^{2}     
\end{equation*}
defined by 
\begin{equation*}
 \underline{x}\longmapsto (x_{0}\alpha_{0}+x_{1}\alpha_{1}+x_{2}\alpha_{2}, x_{t-1}\alpha_{0}+x_{t}\alpha_{1}+x_{t+1}\alpha_{2}).   
\end{equation*}

\textbf{Claim:}
The induced map $\mathbb{P}_{\mathbb{P}^2}(E_{1})\xlongrightarrow{\pi_{1}} \mathbb{P}^{t+1}$ is the blow-up along $V$, where $V$ is as in the statement of the theorem.

\begin{proof}
This result is essentially the main content of \cite{Ra}. This also follows from  \cite[Theorem 1.1 and Proposition 1.3]{ESB}, as follows. 

Let $Q_{0} = X_{1}X_{t+1} - X_{2}X_{t}$, $Q_{1} = X_{2}X_{t-1} - X_{0}X_{t+1}$, and $Q_{2} = X_{0}X_{t} - X_{1}X_{t-1}$. Thus, $V = Z(Q_{0}, Q_{1}, Q_{2})$. It can be easily seen that $V$ is smooth, and $V$ is the same as in the statement of Theorem \ref{theorem:theorem nav}. By $(12)$ (applied to $E_1$ instead of $E$), we see that $\pi$ is an isomorphism outside $V$, and each fibre of $\pi$ over $V$ is $\mathbb{P}^1$. Now \cite[Theorem 1.1 and Proposition 1.3]{ESB} completes the proof.
\end{proof}

\hspace{10pt}Now, Lemma \ref{theorem:sumtriv}, implies that, up to an automorphism of $\mathbb{P}^{r+1}, \mathbb{P}_{\mathbb{P}^2}(E)\xlongrightarrow{\pi}\mathbb{P}^{r+1}$ is birational with fundamental locus the cone over $V$. 
\end{proof}
\section{Geometric Construction of rooftop flips} The main theorem of this section is the following.
\begin{theorem}

In the same hypothesis and notations of Theorem \ref{theorem:theorem drum}, let $W_{\pm}=\mathbb{P}_{Y_{\pm}}(\mathcal{O}\oplus \mathcal{E}_{\pm}(L_{\pm}))$, $W=\mathbb{P}_Y(\mathcal{O}\oplus \mathcal{L}_{+}\otimes \mathcal{L}_{-}),$ $W_0=$ normalization of projective cone over the embedding $i:Y\xhookrightarrow{p_-\times p_+}Y_-\times Y_+\hookrightarrow \mathbb{P}(H^0(Y_-,L_-)\otimes H^0(Y_+,L_+)).$ Then there is a rooftop flip modeled by $Y$ as in \cite[Definition 3.2]{LF}, with $W, W_{\pm}, W_0$ as above. If dim$Y_-> $dim$Y_+$, $\psi$ is a flip. If dim$Y_-= $dim$Y_+$, $\psi$ is a flop.
 \end{theorem}
\begin{proof}
Note that $i^*\mathcal{O}(1)=\mathcal{L}_{+}\otimes \mathcal{L}_{-}.$ Call this line bundle $\mathcal{O}_Y(1).$ Let $W_1$ be the projective cone over the projective embedding $i$. So, $W=\mathbb{P}_Y(\mathcal{O}_Y\oplus\mathcal{O}_Y(1))$ is the blow-up of $W_1$ at the vertex $z_1$ of the cone. Let $q_1:W\to W_0$ be the blow-up map. We regard $Y$ as a subvariety of $W$ via the zero section of $\mathcal{O}_Y(-1)$. So $q_1$ contracts $Y$ to $z_1$. Since $q_1$ has connected fibres, the Stein factorization of $q_1$ is $W\xlongrightarrow{q} W_0\to W_1$, where the last map is the normalization map. If $z_0\in W_0$ is the preimage of $z_1$ under the normalization, then $q$ contracts $Y$ to $z_0.$

By Lemma \ref{lemma:lemma relative}, we get $\textnormal{Bl}_{Y_-}W_-=\textnormal{Bl}_{Y_-}\mathbb{P}_{Y_{-}}(\mathcal{O}\oplus \mathcal{E}_{-}(L_{-}))=\mathbb{P}_{\mathbb{P}(\mathcal{E}_-(L_-))}(\mathcal{O}\oplus\mathcal{O}_{\mathbb{P}(\mathcal{E}_-(L_-))}(1)) = \mathbb{P}_{\mathbb{P}(\mathcal{E}_-)}(\mathcal{O}\oplus\mathcal{O}_{\mathbb{P}(\mathcal{E}_-)}(1)\otimes \mathcal{L}_-)=\mathbb{P}_Y(\mathcal{O}\oplus \mathcal{L}_{+}\otimes \mathcal{L}_{-})=W$. Here we are using: $\mathbb{P}_{Y_{-}}(\mathcal{E_{-}})= Y$ and $\mathcal{O}_{\mathbb{P}(\mathcal{E_-})}(1)= \mathcal{L}_{+}$, which we observed in the proof of Theorem \ref{theorem:theorem drum}. Let $b_-:W\to W_-$ be the blow-up map along $Y_-$. Also by Lemma \ref{lemma:lemma relative}, Ex($b_-)=Y$ and $b_-|_Y:Y\to Y_-$ is same as $p_-$.  Similarly, $\textnormal{Bl}_{Y_+}W_+=W$. If $b_+:W\to W_+$ is the blow-up map along $Y_+$, then Ex($b_+)=Y$ and $b_+|_Y:Y\to Y_+$ is same as $p_+$.  Any curve in $W$ contracted by $b_{\pm}$ lies $Y$, so is contracted by $q$. So there is induced $s_{\pm}: W_{\pm}\to W_0$, contracting $Y_{\pm}$ to $z_0$ such that the following diagram commutes:
 \[
\begin{tikzcd} 
{}& W \arrow[dl, "b_{-}"'] \arrow[dr, "b_{+}"] & \\
W_{-} \arrow[dr, "s_{-}"] &&  W_{+}\arrow[dl, "s_{+}"] \\
  & W_{0}  
\end{tikzcd}
\]
Also, $W_{\pm}$ are not isomorphic over $W_0$, as $b_{\pm}$ are different contractions. Note that $s_{\pm}$ are small contractions as codim$(Y_{\pm}, W_{\pm})>$ codim$(Y, W)=1$. This proves the theorem, except the last statement.

Let $k_{\pm}=$dim$Y-$dim$Y_{\pm}$. For a vector bundle $F$ on $Y_{\pm}$, let deg$F$ be the unique integer such that det$F=$deg$F$. $L_{\pm}$. So index($Y_{\pm})=-$deg$(K_{Y_{\pm}})$. By the formula for the canonical divisor of a projective bundle (for example, as in \cite{Re}, p. 349), $$K_Y=(\text{deg}(\mathcal{E}_+)-\text{index}(Y_-))\mathcal{L}_--(k_-+1)\mathcal{L}_+=(\text{deg}(\mathcal{E}_+)-\text{index}(Y_+))\mathcal{L}_+-(k_++1)\mathcal{L}_-.$$ Here we are using additive notation for tensor product of line bundles. Since Pic($Y$) is freely generated by $\mathcal{L}_-$ and $\mathcal{L}_+$, equating coefficients we get 
\begin{equation}
deg(\mathcal{E}_-)-index(Y_-)=-(k_++1).
\end{equation}

Let $N_{Y_-/W_-}\cong \mathcal{E}^*(-L_-)$ be the normal bundle of $Y_-$ in $W_-$. By adjunction, $$-\text{index}(Y_-)=\text{deg}(K_{W_-}|_{Y_-})+\text{deg}N_{Y_-/W_-}$$ $$=\text{deg}(K_{W_-}|_{Y_-})+\text{deg}(\mathcal{E}_-^*(-L_-))=\text{deg}(K_{W_-}|_{Y_-})-\text{deg}\mathcal{E}_--(k_-+1).$$ Now eq.$(8)$ gives $$\text{deg}(K_{W_-}|_{Y_-})=k_--k_+=\text{dim}(Y_+)-\text{dim}(Y_-).$$ Similarly, $\text{deg}(K_{W_+}|_{Y_+})=\text{dim}(Y_-)-\text{dim}(Y_+)$. This proves the last statement in the theorem.
\end{proof}
For $Y=\mathbb{P}^m\times \mathbb{P}^n$, we get the classical Atiyah flip \textnormal{\cite[Section $3.1$]{LF}}. We get the following
\begin{corollary}\label{theorem:theorem sm}
For any $m,n\geq 1,\, \mathbb{P}_{\mathbb{P}^{n}}(\mathcal{O}_{\mathbb{P}^{n}}\oplus\mathcal{O}_{\mathbb{P}^{n}}(1)^{m+1})$ has another contraction, which is small birational with image the cone over the Segre embedding of $\mathbb{P}^{m}\times\mathbb{P}^{n}$.    
\end{corollary}
\begin{remark}\label{remark:remark flop}
Since $\mathcal{O}_{\mathbb{P}^{n}}\oplus\mathcal{O}_{\mathbb{P}^{n}}(1)^{m+1}$ is globally generated but not ample, the map
\begin{equation*}
\mathbb{P}_{\mathbb{P}^{n}}(\mathcal{O}_{\mathbb{P}^{n}}\oplus\mathcal{O}_{\mathbb{P}^{n}}(1)^{m+1})\xlongrightarrow{\pi}\mathbb{P}(H^{0}(\mathbb{P}^{n}, \mathcal{O}_{\mathbb{P}^{n}}\oplus\mathcal{O}_{\mathbb{P}^{n}}(1)^{m+1}))    
\end{equation*}
is not finite. If $\mathbb{P}_{\mathbb{P}^{n}}(\mathcal{O}_{\mathbb{P}^{n}}\oplus\mathcal{O}_{\mathbb{P}^{n}}(1)^{m+1})\xlongrightarrow{\pi_{1}}Z\xlongrightarrow{g}\mathbb{P} \mathbb(H^{0}(\mathbb{P}^{n}, \mathcal{O}_{\mathbb{P}^{n}}\oplus\mathcal{O}_{\mathbb{P}^{n}}(1)^{m+1}))$ is the Stein factorization of $\pi$, it follows that $\pi_{1}$ is the contraction of $\mathbb{P}_{\mathbb{P}^{n}}(\mathcal{O}_{\mathbb{P}^{n}}\oplus\mathcal{O}_{\mathbb{P}^{n}}(1)^{m+1})$ other than the projection to $\mathbb{P}^{n}$, and Corollary \ref{theorem:theorem sm}, shows $\pi_{1}$ is birational. So, $\pi$ is generically finite onto its image. Applying Lemma \ref{theorem:sumtriv} repeatedly, we see that for all $r\geq 1,\mathbb{P}_{\mathbb{P}^{n}}(\mathcal{O}_{\mathbb{P}^{n}}^{r}\oplus\mathcal{O}_{\mathbb{P}^{n}}(1)^{m+1})\xlongrightarrow{\pi_{r}}\mathbb{P}(H^{0}(\mathbb{P}^{n}, \mathcal{O}_{\mathbb{P}^{n}}^{r}\oplus\mathcal{O}_{\mathbb{P}^{n}}(1)^{m+1}))$ is generically finite but not finite onto its image. By looking at the Stein factorization of $\pi_{r}$, we see that for all $r$, $\mathbb{P}_{\mathbb{P}^{n}}(\mathcal{O}_{\mathbb{P}^{n}}^{r}\oplus\mathcal{O}_{\mathbb{P}^{n}}(1)^{m+1})$ has a birational contraction. 
\end{remark}
\begin{remark}
    Smooth projective varieties of Picard rank $2$ having two projective bundle structures are studied in several papers, for example, \textnormal{\cite{Sa}}, \textnormal{\cite{BSV}}, \textnormal{\cite{Ka}}, \textnormal{\cite{OCR}}. By Theorems \ref{theorem:theorem drum}, \ref{theorem:theorem sm}, for each smooth projective variety $Y$ of Picard rank $2$ having two projective bundle structures such that the associated drum is smooth (see remark $2.17$ in \cite{LF} for the known examples till now), we get examples of
    \begin{enumerate}
        \item Smooth projective variety of Picard rank $2$ having a projective bundle and a smooth blow-up structure, and
        \item Rooftop flip modelled by $Y$.
    \end{enumerate}
\end{remark}
\section{Globally generated vector bundles on projective space}

We finish the study of extremal contractions of projectivizations of vector bundles in Theorem \ref{theorem:theorem cl}.
\begin{enumerate}
\item {Case (1)\label{case:case 1}:} {We can identify $\mathbb{P}_{\mathbb{P}^{m}}(\mathcal{O}_{\mathbb{P}^{m}}^{r}\oplus\mathcal{O}_{\mathbb{P}^{m}}(1))= \textnormal{Bl}_\textnormal{M}\mathbb{P}^{{m+r}}$}, where $M = \{[\underline{x}] \in \mathbb{P}^{m+r}\big{|}\, x_{0} = x_{1} =\ldots  =  x_{m} = 0\}$. If $X\hookrightarrow\mathbb{P}^{m}$ is a subvariety, $Z = \mathbb{P}_{X}(\mathcal{O}_{X}^{r}\oplus \mathcal{O}_{X}(1))\hookrightarrow\mathbb{P}_{\mathbb{P}^{m}}(\mathcal{O}_{\mathbb{P}^{m}}^{r}\oplus \mathcal{O}_{\mathbb{P}^{m}}(1)) = \textnormal{Bl}_\textnormal{M}(\mathbb{P}^{m+r})$ is the strict transform of $p(Z)$, where $p: \textnormal{Bl}_\textnormal{M}(\mathbb{P}^{m+r})\longrightarrow\mathbb{P}^{m+r}$ is the projection. But $p(Z) = C(X)$, the cone over $X$ with vertex $M$. The projective bundle $\mathbb{P}_{X}(\mathcal{O}_{X}^{r}\oplus \mathcal{O}_{X}(1))$ is the blow-up of $C(X)$ at $M$.

Taking $X\hookrightarrow\mathbb{P}^{m}$ to be the $2$-uple embedding of $\mathbb{P}^n$ in $\mathbb{P}^{{n+2 \choose 2}-1}$, we see that $\mathbb{P}_{\mathbb{P}^{n}}(\mathcal{O}_{\mathbb{P}^{n}}^{r-1}\oplus\mathcal{O}_{\mathbb{P}^{n}}(2))$ is the blow-up of $C(X)\hookrightarrow\mathbb{P}^{{n+2 \choose 2}+r-2}$ at the vertex of the cone, which is a linear subspace of dimension $r-2$.

\item {Case (4):} We saw in the proof of corollary \ref{theorem:theorem r} $(2)$ that
\begin{equation*}
\mathbb{P}_{\mathbb{P}^{3}}(\Omega_{\mathbb{P}^3}(2))\xlongrightarrow{\pi}\mathbb{P}(H^{0}(\mathbb{P}^{3}, \Omega_{\mathbb{P}^3}(2)))
\end{equation*}

identifies $\mathbb{P}_{\mathbb{P}^{3}}(\Omega_{\mathbb{P}^3}(2))$ as a projective bundle over a smooth quadric in $\mathbb{P}^{5}$ (Plucker embedding of $\text{Gr}(2,4)$). So, dim$\,\pi^{-1}(x)>0$ for a  general $x$ in Im($\,\pi$). Applying Lemma \ref{theorem:sumtriv} repeatedly, we get dim$\,\tilde{\pi}^{-1}(x)>0$, for a general member in Im$\,\tilde{\pi}$, where $$ \tilde{\pi} : \mathbb{P}_{\mathbb{P}^{3}}(\Omega_{\mathbb{P}^3}(2)\oplus\mathcal{O}_{\mathbb{P}^3}^{r-3})\longrightarrow \mathbb{P}(H^{0}(\mathbb{P}^{3}, \Omega_{\mathbb{P}^3}(2)\oplus\mathcal{O}_{\mathbb{P}^3}^{r-3})) = \mathbb{P}(k^{r-3}\oplus H^{0}(\mathbb{P}^{3}, \Omega_{\mathbb{P}^3}(2))).$$ So, $\tilde{\pi}$ is never birational. A similar conclusion holds for $\mathcal{N}_{\mathbb{P}^{3}}(1).$

\item{Case (6):} $\pi: \mathbb{P}_{\mathbb{P}^n}(E) \hookrightarrow\mathbb{P}(\mathcal{O}_{\mathbb{P}^n}^{r+1})\longrightarrow\mathbb{P}^{r}$ is a morphism to a lower-dimensional variety, so $\pi$ is never birational. The general fibre of $\pi$ is a smooth quadric in $\mathbb{P}^{n}$, which is connected. Thus, $\pi$ is the other contraction of $\mathbb{P}_{\mathbb{P}^n}(E)$.
\end{enumerate}

\begin{theorem}\label{theorem:theorem sbu}
Let $n\geq 2$. If $E$ is a globally generated vector bundle on $\mathbb{P}^n$ of rank greater than or equal to 2 and $c_{1}(E) = 2$, then $\mathbb{P}_{\mathbb{P}^n}(E)$ has a smooth blow-up structure if and only if one of the following holds:
\begin{itemize}
   
\item[(i)] {$E\cong \mathcal{O}_{\mathbb{P}^{n}}(1)\oplus T_{\mathbb{P}^n}(-1)$}. In this case, $\mathbb{P}_{\mathbb{P}^n}(E)$  is the blow-up of a smooth quadric in $\mathbb{P}^{2n+1}$ along a linear subvariety of dimension $n$.
\item[(ii)] {There is a short exact sequence $0\longrightarrow\mathcal{O}_{\mathbb{P}^{n}}\longrightarrow\mathcal{O}_{\mathbb{P}^{n}}(1)\oplus T_{\mathbb{P}^{n}}(-1)\longrightarrow E\longrightarrow 0$}. In this case, $\mathbb{P}_{\mathbb{P}^n}(E)$  is the blow-up of a smooth quadric in $\mathbb{P}^{2n}$ along a linear subvariety of dimension $n-1$.
\item[(iii)] If $n = 2$, there is a short exact sequence $0\longrightarrow \mathcal{O}_{\mathbb{P}^{n}}(-1)^{2}\xlongrightarrow{s}\mathcal{O}_{\mathbb{P}^{n}}^{r+2}\longrightarrow E\longrightarrow 0$ and $E$ does not contain a trivial line bundle as a direct summand. In this case, $2\leq r\leq 4$, and $\mathbb{P}_{\mathbb{P}^n}(E)$ is the blow-up of $\mathbb{P}^{r+1}$ along $V$, where
\begin{equation*}
V = \begin{cases}
\textnormal{twisted cubic}, & \quad \textnormal{if}\quad r = 2, \\
Z(X_{0}X_{3} - X_{1}X_{2}, X_{1}X_{4} - X_{2}X_{3}, X_{0}X_{4} - X_{2}^{2})\hookrightarrow \mathbb{P}^{4}, & \quad\textnormal{if}\quad r = 3, \\
\textnormal{image of Segre embedding}\, \mathbb{P}^{1}\times\mathbb{P}^{2}\hookrightarrow\mathbb{P}^{5}, & \quad \textnormal{if}\quad r = 4.
\end{cases}
\end{equation*} 
\end{itemize}
\end{theorem}
\begin{proof}
The proof follows from Theorems \ref{theorem:theorem nav}, \ref{theorem:theorem r},  \ref{theorem:theorem sm} and the analysis of the cases at the beginning of this section.
\end{proof}
\begin{theorem}\label{theorem:theorem noample}
Let $n\geq 2$, and let $E\neq \mathcal{O}_{{\mathbb{P}^n}}(1)^{\oplus2}$ be a gloobally generated vector bundle of rank $\geq2$ and first Chern class $2$. Then the linear system $|\mathcal{O}_{\mathbb{P}(E)}(1)|$ gives a contraction of $\mathbb{P}_{\mathbb{P}^n}(E)$. In particular, $E$ is nef but not ample.
\end{theorem}
\begin{proof}
Follows from Theorems \ref{theorem:sumtriv}, \ref{theorem:theorem nav}, \ref{theorem:theorem r}, \ref{theorem:theorem sm}, and the analysis of the cases at the beginning of this section.
\end{proof}
We call a vector bundle $E$ over $\mathbb{P}^n$ \textit{big} if $\mathcal{O}_{\mathbb{P}(E)}(1)$ is a big line bundle on $\mathbb{P}_{\mathbb{P}^n}(E).$ .
\begin{theorem}\label{theorem:theorem O}
Let $n\geq 2$, and $E\neq \mathcal{O}_{\mathbb{P}^n}(1)^{\oplus2}$ be a gloobally generated vector bundle of rank $\geq2$ and first Chern class $2$. Then the following are equivalent:
\begin{enumerate}
\item [(i)] $\mathbb{P}_{\mathbb{P}^n}(E)$ has a birational contraction.
\item [(ii)] $E$ is big.
\item [(iii)]  $E$ is of type $(1), (2),$ or $ (3)$ of Theorem \textnormal{\ref{theorem:theorem cl}} or $n=2$ and $E$ is of type $(5)$ of Theorem \textnormal{\ref{theorem:theorem cl}}.
\end{enumerate} 
\end{theorem}
\begin{proof}
The equivalence of $(i)$ and $(ii)$ follows from Theorem \ref{theorem:theorem noample}. The equivalence of $(i)$ and $(iii)$ follows from Lemma \ref{theorem:sumtriv}, Theorems \ref{theorem:theorem nav}, \ref{theorem:theorem r}, \ref{theorem:theorem sm}, and the analysis of the cases at the beginning of this section.
\end{proof}
\hspace{10pt} Now we compute the nef and pseudoeffective cones of $\mathbb{P}_{\mathbb{P}^n}(E)$.

\begin{theorem}\label{theorem:theorem cone}
Let $n\geq 2$, and $E\neq \mathcal{O}_{\mathbb{P}^n}(1)^{\oplus2}$ be a gloobally generated vector bundle of rank $\geq2$ and first Chern class $2$.
\begin{enumerate}
\item [(i)]\textnormal{Nef}$(\mathbb{P}_{\mathbb{P}^n}(E))=<H, \mathcal{O}_{\mathbb{P}(E)}(1)>$
     \item [(ii)] $\overline{\textnormal{Eff}}^{1}(\mathbb{P}_{\mathbb{P}^{n}}(E)) = \langle H, \mathcal{O}_{\mathbb{P}(E)}(1)-cH \rangle $, where $H$ is the pullback of $ \mathcal{O}_{\mathbb{P}^n}(1)$ to $ \mathbb{P}_{\mathbb{P}^{n}}(E)$, and 
$$
c=
\begin{cases}
0 , & \quad \textrm{if E is of type (4)}, n\geq 2\, \textrm{or of type (5)}, n\geq 3\, \textrm {or of type (4)}, \\
\frac{1}{2}, & \quad\textrm{if n=2, E is of type (5)}, \\
1, & \quad \textrm{if E is of type (2) or (3) },\\
2, &  \quad \textrm{if E is of type (1)}.

\end{cases}
$$
     \end{enumerate}
 \end{theorem}
 \begin{proof}
  Part $(i)$ follows easily from Theorem \ref{theorem:theorem noample}. We now prove $(ii)$. If $E$ is of type $(4)$, $ n \geq 2$, or of type $(5)$, $ n \geq 3$, or of type $(4)$, by Theorem \ref{theorem:theorem cone} $ \mathcal{O}_{\mathbb{P}(E)}(1)$ is not big, so $ \overline{\textnormal{Eff}}^{1}(\mathbb{P}_{\mathbb{P}^{n}}(E))= \langle H, \mathcal{O}_{\mathbb{P}(E)}(1) \rangle $, that is, $c=0$.

If $E$ is of type $(1)$, 
$$ H^{0}(\mathbb{P}_{\mathbb{P}^{n}}(E), \mathcal{O}_{\mathbb{P}(E)}(1)^{a}\otimes H^{-b}) = H^{0}(\mathbb{P}^n, \textrm{Sym}^{a}(E)(-b)) $$
\hspace{225pt}$= H^{0}(\mathbb{P}^n, \textrm{Sym}^{a}(\mathcal{O}_{\mathbb{P}^n}^{r-1}\oplus \mathcal{O}_{\mathbb{P}^n}(2))(-b))$

\vspace{2pt}
\hspace{223pt}$= \bigoplus_{i}H^{0}(\mathbb{P}^n, \mathcal{O}_{\mathbb{P}^n}(l_{i}-b)) $

for some integers $ 0\leq l_{i} \leq 2a$, and $ l_{i} = 2a$ for some $i$. So we have 
$$ H^{0}(\mathbb{P}_{\mathbb{P}^{n}}(E), \mathcal{O}_{\mathbb{P}(E)}(1)^{a} \otimes H^{-b}) =0  \Leftrightarrow b > 2a .$$
From this, one easily gets 
$$ \overline{\textnormal{Eff}}^{1}(\mathbb{P}_{\mathbb{P}^{n}}(E))= \langle H, \mathcal{O}_{\mathbb{P}(E)}(1) - 2H \rangle .$$
 This implies that $c=2$. A similar proof shows $c=1$ if $E$ is of type $(2)$. If $E$ is of type $(3)$, we observed in Theorem \ref{theorem:theorem 4.10} that $ E = \mathcal{O}_{\mathbb{P}^n}^{s}\oplus E_{1}$ for some $ s\geq 0$ and here either $ E_{1}=\mathcal{O}_{\mathbb{P}^n}(1) \oplus T_{\mathbb{P}^{n}}(-1)$ or $ E_{1}$ fits into a short exact sequence 
 $$ 0\longrightarrow  \mathcal{O}_{\mathbb{P}^{n}} \longrightarrow \mathcal{O}_{\mathbb{P}^{n}}(1) \oplus T(-1) \longrightarrow E_{1} \longrightarrow 0 .$$ 
 By Lemma \ref{theorem:theorem osum}, we can assume without loss of generality $ E= E_{1}.$ By Theorem \ref{theorem:theorem 4.10}, $ \mathbb{P}_{\mathbb{P}^{n}}(E)$ has a smooth blow-up structure, call the exceptional divisor $E_{2}$. By Lemma \ref{theorem:theorem E}, $E_{2}$ is effective but not big, so 
$$ \overline{\textnormal{Eff}}^{1}(\mathbb{P}_{\mathbb{P}^{n}}(E))= \langle H, E_{2} \rangle.$$

The surjection $ \mathcal{O}_{\mathbb{P}^n}^{n+1} \longrightarrow T_{\mathbb{P}^n}(-1)$ is Euler exact sequence gives a surjection $ \mathcal{O}_{\mathbb{P}^n}^{n+1} \oplus \mathcal{O}_{\mathbb{P}^n}(1) \longrightarrow E$, which induces a morphism 
$$ \pi: \mathbb{P}_{\mathbb{P}^{n}}(E)\hookrightarrow \mathbb{P}( \mathcal{O}_{\mathbb{P}^n} \oplus \mathcal{O}_{\mathbb{P}^n}(1)) \xlongrightarrow{\pi_{1}} 
\mathbb{P}^{2n+1}.
$$ 

By Lemma \ref{theorem:theorem key}, here $ \pi_{1}$ is the blow-up of $ \mathbb{P}^{2n+1}$ along a linear subspace of dimension $n$, with exceptional divisor $ E_{3}= \mathbb{P}( \mathcal{O}_{\mathbb{P}^n}^{n+1}) \hookrightarrow \mathbb{P}(\mathcal{O}_{\mathbb{P}^n}^{n+1} \oplus \mathcal{O}_{\mathbb{P}^n}(1) ).$ Hence  we have 
$$ E_{3}= -H + \mathcal{O}_{\mathbb{P}(\mathcal{O}_{\mathbb{P}^n}^{n+1} \oplus \mathcal{O}_{\mathbb{P}^n}(1))}(1)$$
 in $ \textnormal{N}^{1}(\mathbb{P}_{\mathbb{P}^{n}}(\mathcal{O}_{\mathbb{P}^n}^{n+1} \oplus \mathcal{O}_{\mathbb{P}^n}(1)))$, $E_{2}= E_{3} \cap \mathbb{P}_{\mathbb{P}^{n}}(E)$, so $ E_{2}= \mathcal{O}_{\mathbb{P}(E)}(1)-H$, as $$ \mathcal{O}_{\mathbb{P}(\mathcal{O}_{\mathbb{P}^n}^{n+1} \oplus \mathcal{O}_{\mathbb{P}^n}(1))}(1)\big{|}_{\mathbb{P}(E)}= \mathcal{O}_{\mathbb{P}(E)}(1).$$ 
 So we conclude that $ c=1$. If $E$ is of type $(5)$ and $n=2$, by Theorem \ref{theorem:theorem nav}, we have $E= \mathcal{O}_{\mathbb{P}^n}^s \oplus E_{1}$, for some $ s\geq 0$, $E_{1}$ does not have trivial bundle as a direct summand,  and  there is a short exact sequence    

$$ 0 \longrightarrow \mathcal{O}_{\mathbb{P}^n}(-1)^{2} \longrightarrow \mathcal{O}_{\mathbb{P}^n}^{r_{1}+1} \longrightarrow E_{1} \longrightarrow 0$$ 
for some $r_{1}$ with $ 2\leq r_{1} \leq 4.$ By the Lemma, we can assume without loss of generality that $ s=0$. So $ r=r_{1}$, $E=E_{1}$. By Theorem \ref{theorem:theorem sbu}, $ \mathbb{P}_{\mathbb{P}^{n}}(E)$ has a smooth blow-up structure, called the exceptional divisor $E_{2}.$ By Lemma \ref{theorem:theorem E}$, E_{2}$ is effective but not big, so $$ \overline{\textnormal{Eff}}^{1}(\mathbb{P}_{\mathbb{P}^{n}}(E))= \langle H, E_{2} \rangle .$$ 
By {\cite [Corollaries 3.12, 3.13, 3.14]{Ra}},  and part $(i)$ of Theorem \ref{theorem:theorem cone}, we have $ \mathcal{O}_{\mathbb{P}(E)}(1)= 2H- E_{2},$ so $E_{2}= 2H- \mathcal{O}_{\mathbb{P}(E)}(1)$ i.e., $ c= \frac{1}{2}$.
\end{proof}


\section{Acknowledgement}
The authors are grateful to Professor Nagaraj D.S. and Professor János Kollár for giving valuable suggestions and references.

\vspace{40pt}

\begin{flushleft}
{\scshape Indian Institute of Science Education and Research Tirupati, Srinivasapuram, Yerpedu Mandal, Tirupati Dist, Andhra Pradesh, India – 517619.}

{\fontfamily{cmtt}\selectfont
\textit{Email address: ashimabansal@students.iisertirupati.ac.in} }
\end{flushleft}
\vspace{0.5mm}
\begin{flushleft}
{\scshape Fine Hall, Princeton, NJ 700108}.

{\fontfamily{cmtt}\selectfont
\textit{Email address: ss6663@princeton.edu} }
\end{flushleft}
\vspace{0.5mm}
\begin{flushleft}
{\scshape Indian Institute of Science Education and Research Tirupati, Srinivasapuram, Yerpedu Mandal, Tirupati Dist, Andhra Pradesh, India – 517619.}

{\fontfamily{cmtt}\selectfont
\textit{Email address: shivamvats@students.iisertirupati.ac.in} }
\end{flushleft}

\end{document}